\documentclass[leqno]{amsart}
\usepackage{amsmath}
\usepackage{amssymb}
\usepackage{amsthm}
\usepackage{enumerate}
\usepackage[mathscr]{eucal}
\usepackage{xcolor}
\theoremstyle{plain}
\usepackage{tikz}
\usepackage{soul}
\usepackage[normalem]{ulem}
\newtheorem{theorem}{Theorem}[section]
\newtheorem{lemma}[theorem]{Lemma}
\newtheorem{prop}[theorem]{Proposition}
\theoremstyle{definition}

\newtheorem{remark}[theorem]{Remark}

\newtheorem{example}[theorem]{Example}

\newtheorem{cor}[theorem]{Corollary}
\theoremstyle{remark}
\begin{document}
	\title[A unified approach to  optimization problems]{ A unified approach to a family of optimization problems in Banach spaces}
\author[K. Paul, S.Roy, D. Sain, and S. Sohel]{Kallol Paul, Saikat Roy, Debmalya Sain and  Shamim Sohel}

\address[Paul]{Department of Mathematics, Jadavpur University, Kolkata 700032, India.}	\email{kalloldada@gmail.com}

\address[Roy]{Department of Mathematics, IIT Bombay, Maharashtra 400076.}  \email{saikatroy.cu@gmail.com}

\address[Sain]{Department of Mathematics\\ Indian Institute of Information Technology, Raichur\\ Karnataka 584135 \\INDIA}	\email{saindebmalya@gmail.com}

\address[Sohel]{Department of Mathematics, Jadavpur University, Kolkata 700032, India.}	\email{shamimsohel11@gmail.com}

	\begin{abstract}
		Our principal aim is to illustrate that the concept Birkhoff-James orthogonality can be applied effectively to obtain a unified approach to a large family of optimization problems in Banach spaces. We study such optimization problems from the perspective of Birkhoff-James orthogonality in certain suitable Banach spaces. In particular, we demonstrate the duality between  the Fermat-Torricelli problem and the Chebyshev center problem which are important particular cases of the least square problem. We revisit the Fermat-Torricelli problem for three and four points and solve it using the same technique. We also investigate the behavior of the Fermat-Torricelli points under the addition or replacement of a new point, and present several new results involving the locations of the Fermat-Torricelli point and the Chebyshev center.
		
	\end{abstract}
	
	\subjclass[{2020}]{Primary 46N10, Secondary 47N10, 51N20, 46B20}
\thanks{Mr Shamim Sohel would like to thank  CSIR, Govt. of India, for the financial support in the form of Senior Research Fellowship under the mentorship of Prof. Kallol Paul.} 
	
	\maketitle
	\section{Introduction}

The least square problem is one of the most important and well-studied problems in optimization theory, especially because of its deep connections with multiple scenarios of practical importance. Let $\mathbb{X}, \mathbb{Z}$ be Banach spaces and let $A$ be a bounded linear operator from $\mathbb{X}$ to $ \mathbb{Z}.$  Then the (linear) least square problem can be stated as:
\begin{equation}\label{least:eqn}
	\underset{x \in \mathbb{X}} {\text{minimize}} \quad  \|A x- z \|,
\end{equation}
where $z \in \mathbb{Z}.$ We refer  the readers to \cite{BV, DMS} for a detailed study on this problem, in the setting of Banach spaces. The purpose of this article is to study the  least square problem from the perspective of Birkhoff-James orthogonality. We investigate two of the most famous location problems falling under this category, namely, the Fermat-Torricelli problem and the Chebyshev center problem, which are extremely important and well-studied from both theoretical and practical viewpoints.\\

		In the $17$th century, Pierre de Fermat proposed the following question: \emph{Given three points in the Euclidean plane, find a point such that the sum of its distances to the three given points is minimum.} The problem was first solved by E. Torricelli and it is aptly called the Fermat-Torricelli problem. Although the original  Fermat-Torricelli problem is concerned  with only three points in the Euclidean plane, there are several natural generalizations of the same.  One such way is to consider not only three but finitely many points in a Banach space. It should be noted that although the classical Fermat-Torricelli problem in the Euclidean plane is completely (and explicitly) solved for the cases of three points and four points, only algorithmic solutions are known for more than four distinct points.  Over the years, many mathematicians have used different techniques to solve the problem and  its variants. We refer the readers to  \cite{BFS,D,F,K,K2,U,Z} and the references therein, for more information in this regard. On the other hand, the Chebyshev center problem for a bounded subset in a Banach space deals with finding the center of the minimal-radius ball containing that particular subset. As in the case of the Fermat-Torricelli problem, the Chebyshev center problem has also been widely studied using various approaches over the years. For the detailed background and other relevant information on the Chebyshev center problem, we refer to \cite{AM, BT, G, SKPR} and the references therein. However, to the best of our knowledge, \emph{no single approach has been proposed to tackle both the Fermat-Torricelli problem and the Chebyshev center problem. In this article, we would like to illustrate that the concept of orthogonality can be applied effectively to study both these problems from a single viewpoint. Moreover, our approach reveals that there is a certain kind of duality between these two problems, which perhaps went unnoticed despite extensive research on both of these two well-known problems in the theory of optimization in Banach spaces.} Let us now mention some basic notations and terminologies.\\

%Although the aid of Birkhoff-James orthogonality is used to approach the least square problem, the main focus of this article is on the Fermat-Torricelli problem and the Chebyshev center problem for finitely many points. We have demonstrated the dual nature of the Fermat-Torricelli problem and the Chebyshev center problem for finitely many points in a Banach space using the Birkhoff-James orthogonality method. 
	%We  provide a solution of  the weighted Fermat-Torricelli problem for three points and the classical problem  for four points in Euclidean plane using the concepts of Birkhoff-James orthogonality. It is worth mentioning here that  we solve the Fermat-Torricelli problem  both for three and four points  using the same technique, perhaps for the first time. We study the nature of the Fermat-Torricelli point under an addition or a replacement of a point.   We also provide a solution for the Chebyshev center for finitely many points in form of an algorithm.  Let us now quickly recall the Fermat-Torricelli problem and the Chebyshev center problem before moving on.

	Let  $ \mathbb{X}, \mathbb{Y} $ be  Banach spaces and let $ B_{\mathbb{X}}= \{ x \in \mathbb{X}: \|x\| \leq 1\} $  and $ S_{\mathbb{X}}= \{ x \in \mathbb{X}: \|x\| = 1\} $ denote  the unit ball and the unit sphere of $\mathbb{X},$ respectively.  The  dual space of $ \mathbb{X}$ is denoted by  $ \mathbb{X}^*.$ For any $f \in \mathbb{X}^*,$ the kernel of $f,$  denoted by $ker f,$ is defined as $ker f=\{x \in \mathbb{X}: f(x)=0\}.$ For a bounded linear operator $A$ from $\mathbb{X}$ to $\mathbb{Y},$ the range of the operator $A,$ to be  denoted by $\mathcal{R}(A),$ is defined as $\mathcal{R}(A)=\{Ax \in \mathbb{Y}: x \in \mathbb{X}\}.$ An element $ f \in S_{\mathbb{X}^*}$ is said to be a supporting functional at $(0 \neq) x \in \mathbb{X}$ if $f(x)= \|x\|.$ The set of all supporting functionals at $x \in S_{\mathbb{X}},$  denoted by $J(x),$ is defined as $J(x)= \{ f \in S_{\mathbb{X}^*}: f(x)= 1\}.$  An element $x  \in S_\mathbb{X}$ is said to be smooth in $ \mathbb{X} $ if $J(x)$ is a singleton. We say that $x \in S_{\mathbb{X}}$ is $k$-smooth if $dim~span ~J(x)=k,$ i.e., if  there exist exactly $k$  linearly independent supporting functionals at $x.$ $\mathbb{X}$ is said to be strictly convex if for any two distinct  elements $x , y \in S_{\mathbb{X}},$ $\|\frac{x+y}{2}\|<1.$ Equivalently, $ \mathbb{X} $ is strictly convex if and only if the unit sphere $ S_{\mathbb{X}} $ contains no non-trivial straight line segment. We use the notation $co(\{z_1, z_2, \ldots, z_n\}) $ to denote the convex hull of the points $ z_1, z_2, \ldots, z_n \in \mathbb{X},$ i.e., 
	\[ co(\{z_1, z_2, \ldots, z_n\}) = \big \{ z \in \mathbb{X} : z = \sum_{i=1}^n t_iz_i,  \sum_{i=1}^n t_i=1, t_i \geq 0 \big \} .\]
	For $z_1,z_2 \in \mathbb{X},$ the  straight line segment joining $z_1,z_2$ is denoted  by $L[z_1,z_2] ,$  i.e., $L[z_1,z_2] =
	\{ (1-t)z_1 + tz_2: 0 \leq t \leq 1 \}.$ 
	Throughout this article, for a Banach space $\mathbb{X},$ we use the notations $\ell_1^n(\mathbb{X})$ and $\ell_{\infty}^n(\mathbb{X})$ to denote the spaces $\underbrace{\mathbb{X} \oplus_1 \mathbb{X} \oplus_1 \ldots \oplus_1 \mathbb{X}}_{n-times}$ and $\underbrace{\mathbb{X} \oplus_\infty \mathbb{X} \oplus_\infty \ldots \oplus_\infty \mathbb{X}}_{n-times},$ respectively. Therefore, for any $\widetilde{x}=(x_1, x_2, \ldots, x_n) \in \ell_1^n(\mathbb{X}),$ it is clear that $ x_i \in \mathbb{X} $ and $\|\widetilde{x}\|_1= \sum_{i=1}^{n} \|x_i\|.$ Similarly, for any $\widetilde{y}=(y_1, y_2, \ldots, y_n) \in \ell_\infty^n(\mathbb{X}),$ we note that $y_i \in \mathbb{X}$ and $\|\widetilde{y}\|_\infty= \max\{ \|y_i\|: 1 \leq i \leq n \}.$\\
	
		The classical Fermat-Torricelli problem for $n$ number of  points $z_1, z_2, \ldots, z_n$ in  a Banach space $\mathbb{X}$ asks to find a point $w$ in $\mathbb{X}$ such that the sum of its distances to the $n$ given points is minimum. This implies that for any point $x$ in $\mathbb{X},$
	\[
	\|z_1 -w\| + \|z_2-w\| + \ldots+ \|z_n-w\| \leq \|z_1- x\| + \| z_2-x\| +\ldots + \|z_n-x  \|.
	\] 
	The point $w$ is called the Fermat-Torricelli point of $z_1, z_2, \ldots, z_n.$ For some fixed positive real numbers $ \alpha_1, \alpha_2, \ldots, \alpha_n,$ we say that $w $ is the weighted Fermat-Torricelli point of $z_1, z_2, \ldots, z_n$ with respective weights $ \alpha_1, \alpha_2, \ldots, \alpha_n$ if
	\[
	\alpha_1	\|z_1 -w\| +  \alpha_2\|z_2-w\| + \ldots+ \alpha_n\|z_n-w\| \leq \alpha_1 \|z_1- x\| + \alpha_2 \| z_2-x\| +\ldots + \alpha_n \|z_n-x  \|,
	\]
	for every element $x \in \mathbb{X}.$  Clearly, the classical Fermat-Torricelli problem is a weighted problem where the weights are all $1.$ 
	On the other hand, for a bounded set $A$ of a Banach space $\mathbb{X},$ $w \in \mathbb{X}$ is said to be Chebyshev center of $A$ if $\sup\{\|z- w\|: z \in A\} \leq \sup \{\|z-x\|: z \in A\},$ for any $x \in \mathbb{X}.$ If $w$ is the Chebyshev center of a bounded set $A$ of $\mathbb{X},$ then $\sup\{\|z-w\|: z\in A\}$ is said to be the Chebyshev radius of $A.$ Let $z_1, z_2, \ldots, z_n \in \mathbb{X}$ and let $\alpha_1, \alpha_2, \ldots, \alpha_n > 0. $ Then we say $w$ is the weighted Chebyshev center of $z_1, z_2, \ldots, z_n$ with respective weights $\alpha_1, \alpha_2, \ldots, \alpha_n,$ if $\sup \{ \alpha_i \| z_i -w\|: 1 \leq i \leq n\} \leq \sup \{ \alpha_i \|z_i-x\|: 1 \leq i \leq n\}.$ \\
	
	The key idea used in this article is the connection shared by Birkhoff-James orthogonality with the least square problem in a Banach space. It should be noted that there are several notions of orthogonality in a Banach space, Birkhoff-James orthogonality \cite{B, J} being arguably the most important and also the most widely studied one among them. For any two elements $ x, y $ in a Banach space $ \mathbb{X}, $  \emph{$ x $ is said to be Birkhoff-James orthogonal to $ y, $} written as $ x \perp_B y, $ if $ \| x+\lambda y \| \geq \| x \| $ for all scalars $\lambda.$  The interplay  between Birkhoff-James orthogonality and the geometric properties of the ground space $\mathbb{X}$ is wonderfully displayed in the pioneering articles \cite{J,J2}. In recent times (see \cite{S21,SR22,SPMR20,SP13}), Birkhoff-James orthogonality has found use in the study of the geometric properties of the space of bounded linear operators on $\mathbb{X}.$   We use the concept of Birkhoff-James orthogonality to completely solve the Fermat-Torricelli problem for three as well as four points. For any $y \in \mathbb{X},$ by $ ^\perp y$ we denote the collection of all elements $ x \in \mathbb{X} $ such that $ x  \perp_B y,$ i.e., $^\perp y= \{x \in \mathbb{X}: x \perp_B y\}.$ Given an element $x \in \mathbb{X}$ and a subspace $\mathbb{Y}$ of $\mathbb{X},$ we say that $y_0$ is a best approximation to $x$ out of $\mathbb{Y}$ if $\|x- y_0\|\leq \|x- y\|,$ for any $y \in \mathbb{Y}$ and in that case, $\|x- y_0\|$ is said to be the distance from $x$ to $\mathbb{Y}.$ It is easy to observe that  $y_0$ is a best approximation to $x$ out of $\mathbb{Y}$ if and only if $(x-y_0) \perp_B \mathbb{Y},$ i.e., $(x-y_0) \perp_B y,$ for any $y \in \mathbb{Y}.$  The readers are referred to \cite{S21, SR22, S70} and the references therein, for several useful applications of best approximation in Banach spaces.\\

	This article is divided into three sections including the introductory one. We first demonstrate that the Fermat-Torricelli  problem and the Chebyshev center problem are two special types of least square problem. Of course, this is well-known and is included mainly for the sake of completeness, apart from the fact that the pivotal role played by orthogonality in the whole scheme sets the tone of the article from the very beginning. Our main goal is to establish the duality between the Fermat-Torricelli  problem and the Chebyshev center problem for the weighted cases, and to present a unified approach to the respective solutions. The main results of this article are divided into two subsections, namely, Section-I and Section-II. In Section-I, we solve the weighted Fermat-Torricelli problem  for three points. As an immediate consequence of our exploration, the solution to the classical Fermat-Torricelli problem for three points follows directly. In this section we also solve the  classical Fermat-Torricelli problem for four points using the same technique. Moreover, we examine the behavior of the Fermat-Torricelli point of a given system, under the addition  or replacement of a point in the weighted settings.  Using this concept, we characterize the cases where the Fermat-Torricelli point is actually one of the given points for which we seek to determine the solution. In Section-II, we study the Chebyshev center problem and provide an algorithmic solution towards finding the Chebyshev center for finitely many points in the Euclidean plane. We also discuss the cases where the Fermat-Torricelli point and the Chebyshev center coincides for a system consisting of either three or four points. We would like to emphasize that while some of these results were known previously, we also obtain some new results which provide additional insights towards a better understanding of both these problems from a singular perspective. We end this section with  the following fundamental observation which will be essential for our study.
	
		\begin{theorem}\cite{J}\label{James}
		Let $\mathbb{X}$ be a Banach space and let $ x, y \in \mathbb{X}.$ Then $x \perp_B y$ if and only if there exists $ f \in J(x) $ such that $ f(y)=0. $
	\end{theorem}

\section{Connection between the least square problem, the Fermat-Torricelli problem and the Chebyshev center problem}
	We begin with a basic observation that illustrates the connection between  the least square problem in a Banach space described in  (\ref{least:eqn}) and the notion of Birkhoff-James orthogonality.
		Let $\mathbb{X}, \mathbb{Z}$ be Banach spaces and let $A$ be a bounded linear operator from $\mathbb{X}$ to $\mathbb{Z}.$ Suppose that $z \in \mathbb{Z}.$ Assume that $w \in \mathbb{X}$ satisfies the following:
		\[
		\|Aw- z\|=\underset{x \in \mathbb{X}} {\text{min}} \quad  \|A x- z \|.
		\]
	This is equivalent to 
	\begin{eqnarray*}
		\|Aw-z\| &\leq& \|u-z\|, \text{ for any} \, \, u \in \mathcal{R}(A),\\
		i.e., \|z- Aw\| & \leq & \|z -Aw + u\|, \text{ for any} \, \, u \in \mathcal{R}(A).
	\end{eqnarray*}
Therefore, the following result is immediate from the above inequalities and the definition of best approximations in a Banach space.

	\begin{prop}\label{prop:least}
		Let $\mathbb{X}, \mathbb{Z}$ be Banach spaces and let $A$ be a bounded linear operator from $\mathbb{X}$ to $\mathbb{Z}.$ Suppose that $z \in \mathbb{Z}.$  Then the following conditions are equivalent:
		\begin{itemize}
			\item[(i)] $\|Aw-z\|= 	\underset{x \in \mathbb{X}} {\text{min}} \quad  \|A x- z \|,$ for some $w \in \mathbb{X}$
			
			\item[(ii)] $(Aw- z) \perp_B \mathcal{R}(A)$
			
			\item[(iii)]    $Aw$ is a best approximation to $z$ out of $\mathcal{R}(A).$
		\end{itemize}
	\end{prop} 
		
	Let us now show that the Fermat-Torricelli problem is indeed a  least square problem. Let $z_1, z_2, \ldots, z_n \in \mathbb{X}$ and  let $\alpha_1, \alpha_2, \ldots, \alpha_n > 0. $ In the least square problem described in (\ref{least:eqn}), take $\mathbb{Z}=  \ell_1^n(\mathbb{X})$ and $z= (\alpha_1 z_1,\alpha_2 z_2, \ldots, \alpha_n z_n) \in \ell_1^n(\mathbb{X}).$ Consider the linear operator $A: \mathbb{X} \to \ell_1^n(\mathbb{X})$ given by $Ax=(\alpha_1x, \alpha_2x, \ldots, \alpha_nx),$ for every $x \in \mathbb{X}.$   Then  the problem (\ref{least:eqn})  reduces to:
	\[
	\underset{x \in \mathbb{X}}{ \text{minimize}} \|(\alpha_1 x, \alpha_2 x, \ldots, \alpha_n x)- z\|_1.
	\]	
 Suppose that $w \in \mathbb{X}$ is a solution to the above minimization problem. Then we have  that for any $x \in \mathbb{X},$
			$$	\|( \alpha_1w, \alpha_2w, \ldots, \alpha_nw)-z\|_1 \leq \|(\alpha_1x, \alpha_2 x, \ldots, \alpha_n x)- z\|_1 $$ 
			and so, 
\[	\alpha_1	\|z_1-w\|+ \alpha_2\|z_2-w\|+\ldots+ \alpha_n\|z_n-w\| \leq \alpha_1\|z_1-x\|+\alpha_2\|z_2-x\|+\ldots+ \alpha_n\|z_n-x\|.\]
This shows that $w$ is the weighted Fermat-Torricelli point of $z_1, z_2, \ldots, z_n$ with the respective weights $\alpha_1, \alpha_2, \ldots, \alpha_n.$ Thus  the Fermat-Torricelli problem is a special type of the least square problem.  Next we express the Fermat-Torricelli problem in terms of Birkhoff-James orthogonality.  

	\begin{prop}\label{Fermat}
	Let $\mathbb{X}$ be a Banach space and let $z_1, z_2, \ldots, z_n \in \mathbb{X}.$ Suppose that $ \alpha_1, \alpha_2, \ldots, \alpha_n > 0 $ and $\mathbb{Y}=\{(\alpha_1x, \alpha_2 x, \ldots, \alpha_n x): x \in \mathbb{X}\}$ is a subspace of $\ell_1^n(\mathbb{X}).$ Let $ z =  (\alpha_1z_1, \alpha_2z_2, \ldots, \alpha_nz_n),\widetilde{w} = (\alpha_1w,\alpha_2w, \ldots, \alpha_nw)  \in \ell_1^n(\mathbb{X}).$   Then the following conditions are equivalent:
	\begin{itemize}
		\item[(i)]  $w \in \mathbb{X}$ is the weighted Fermat-Torricelli point of $z_1, z_2, \ldots, z_n$ with respective weights $\alpha_1, \alpha_2, \ldots, \alpha_n$ 
		
		\item[(ii)] 	$ (z - \widetilde{w}) \perp_B \mathbb{Y}, $ in the Banach space $\ell_1^n(\mathbb{X})$
		
		\item[(iii)] $\widetilde{w}$ is the best approximation to $z$ out of $\mathbb{Y},$ in the Banach  space $\ell_1^n(\mathbb{X}).$ 
	\end{itemize}
\end{prop}  
		
 Analogously, it is rather easy to observe that the Chebyshev center problem for finitely many points is also a least square problem. Let $z_1, z_2, \ldots, z_n \in \mathbb{X} $  and let $\alpha_1, \alpha_2, \ldots, \alpha_n > 0. $ In the least square problem described in (\ref{least:eqn}), take $\mathbb{Z} = \ell_\infty^n(\mathbb{X})$ and  let $z= (\alpha_1 z_1,\alpha_2 z_2, \ldots, \alpha_n z_n) \in  \ell_\infty^n(\mathbb{X}).$  Consider the linear operator $A: \mathbb{X} \to \ell_\infty^n(\mathbb{X})$ given by $Ax=(\alpha_1x, \alpha_2x, \ldots, \alpha_nx),$ for every $x \in \mathbb{X}.$   Then the problem (\ref{least:eqn})  reduces to:
  \[
  \underset{x \in \mathbb{X}}{ \text{minimize}} \|(\alpha_1 x, \alpha_2 x, \ldots, \alpha_n x)- z\|_{\infty}.
  \]	
   Suppose that $w \in \mathbb{X}$ satisfies the above minimization problem. Then we have that for any $x \in \mathbb{X},$
  $$	\|( \alpha_1w, \alpha_2w, \ldots, \alpha_nw)-z\|_{\infty} \leq \|(\alpha_1x, \alpha_2 x, \ldots, \alpha_n x)- z\|_{\infty}$$ 
  and so, 
  \[ \max \{	\alpha_i	\|z_i-w\|: 1 \leq i \leq n  \} \leq  \max \{\alpha_i\|z_i-x\|:  1 \leq i \leq n\}.\]
  Thus $w$ is the weighted Chebyshev center of $z_1, z_2, \ldots, z_n$ with respective weights $\alpha_1, \alpha_2, \ldots, \alpha_n.$ This establishes that the Chebyshev center problem is also a special type of least square problem. Next  we express the Chebyshev center problem in terms of Birkhoff-James orthogonality.

	\begin{prop}\label{Cheby}
		Let $\mathbb{X}$ be a Banach space and let $z_1, z_2, \ldots, z_n \in \mathbb{X}.$ Suppose that $ \alpha_1, \alpha_2, \ldots, \alpha_n > 0 $ and $\mathbb{Y}=\{(\alpha_1x, \alpha_2 x, \ldots, \alpha_n x): x \in \mathbb{X}\}$ is a subspace of $\ell_\infty^n(\mathbb{X}).$ Let $ z =  (\alpha_1z_1, \alpha_2z_2, \ldots, \alpha_nz_n),\widetilde{w} = (\alpha_1w,\alpha_2w, \ldots, \alpha_nw)  \in \ell_\infty^n(\mathbb{X}).$   Then the following conditions are equivalent:
		\begin{itemize}
			\item[(i)]  $w \in \mathbb{X}$ is the weighted Chebyshev center  of $z_1, z_2, \ldots, z_n$ with respective weights $\alpha_1, \alpha_2, \ldots, \alpha_n$ 
			
			\item[(ii)] 	$ (z - \widetilde{w}) \perp_B \mathbb{Y} ,$ in the Banach space $\ell_\infty^n(\mathbb{X})$
			
			\item[(iii)] $\widetilde{w}$ is the best approximation to $z$ out of $\mathbb{Y},$ in the Banach  space $\ell_\infty^n(\mathbb{X}).$ 
		\end{itemize}
	\end{prop}

Observe  that whenever $ \mathbb{X} $ is finite-dimensional (such as the classical Fermat-Torricelli problem), it can be identified with $ \mathbb{X}^* $ \emph{upto vector space structure.} Therefore, by virtue of the well-known fact that $ \ell_1^n(\mathbb{X})^* = \ell_\infty^n(\mathbb{X}^*), $ the above problems are indeed dual in nature. From Proposition \ref{Fermat} and Proposition \ref{Cheby},  it follows that the Fermat-Torricelli problem for $n$-number of points in $\mathbb{X}$ is an orthogonality problem in the space $\ell_1^n(\mathbb{X}),$ whereas the Chebyshev center problem  for $n$-number of points in $\mathbb{X}$ is the same orthogonality problem in the space $\ell_\infty^n(\mathbb{X}).$ For the classical versions of the Fermat-Torricelli problem and the Chebyshev problem in the complex plane, this duality is even more clear. Indeed, in this case, $ \mathbb{X} = \mathbb{X}^* = \mathbb{C}. $ Consequently, the two problems are exactly dual to one another. It should be further noted that whenever $\mathbb{X}$ is finite-dimensional, the existence of the Fermat-Torricelli point and the Chebyshev  center  is guaranteed from the fact that the best approximation to any point out of a finite-dimensional subspace always exists.

	\section{Main Results.}
	
	In this section we study the Fermat-Torricelli problem and the Chebyshev center problem in the complex plane.  For the rest of the article, the spaces $\ell_1^n$ and $\ell_{\infty}^n$ are considered over the complex field $\mathbb{C}.$

	\section*{Section-I}
	
We begin this section with a study of the Fermat-Torricelli problem for three and four points in the complex plane. Although the corresponding solutions are previously known, it is worth mentioning here that  we solve the Fermat-Torricelli problem for both three and four points \emph{using the same technique, perhaps for the first time}. Previously obtained solutions, as can be seen in \cite{GT, HZ, P, S96, T, U, Z2}, use different methods for these cases. It is worth mentioning that in \cite{T}, the Fermat-Torricelli point has been characterized for $n$-number of points, but the solution is far from explicitly determining the exact location of the desired point. On the other hand, our approach solves the problem explicitly for both three points and four points in similar fashion. Furthermore, we also provide a complete characterization of the Fermat-Torricelli point for $n$-number of points, in terms of Birkhoff-James orthogonality. In particular, this also illustrates the importance of the concept of Birkhoff-James orthogonality towards a unified approach for such kind of problems.\\

	For any three distinct complex numbers $u, v, w$ we 
	use the notation $\angle (\overrightarrow{uv}, \overrightarrow{uw})$ to denote  the angle at $u$ formed by the line segments $L[u,v]$ and $L[u, w],$
	measured from $L[u,v]$ to $L[u, w]$		in the positive direction. In particular, $\angle (\overrightarrow{uv}, \overrightarrow{uw})= \theta$ implies that 
	$(u-w)= \lambda(u-v)e^{i\theta},$ for some $\lambda> 0.$ It is easy to observe that if $\angle (\overrightarrow{uv}, \overrightarrow{uw})= \theta$ then $\angle (\overrightarrow{uw}, \overrightarrow{uv})= 2\pi-\theta.$ 
	Moreover, for any two $ v, w \in \mathbb{C},$ we use the notation $\angle (\overrightarrow{v}, \overrightarrow{w})$ to denote the angle at 
	$0 \in \mathbb{C}$ formed by the line segments $L[0,v]$ and $ L[0,w],$ measured from $L[0,v]$ to $L[0,w]$ in the positive direction.
	%	We use the notation $\langle [u,v], [u,w] \rangle= \theta$ to denote that  $(u-w)= \lambda(u-v)e^{i\theta},$ for some $\lambda> 0.$ It is easy to observe that if $\langle [u,v], [u,w] \rangle= \theta$ then $\langle [u,w], [u,v] \rangle= \pi-\theta.$ 
	Clearly,  $u, v, w$ are collinear whenever $\angle (\overrightarrow{uv}, \overrightarrow{uw})=  n\pi,$ for $n  = 0,1,2, \ldots.$
		Using  Theorem \ref{James} and Proposition \ref{Fermat},   it follows  that $ w \in \mathbb{C}$ is the weighted Fermat-Torricelli point for the $n$ complex numbers $z_1,z_2, \ldots, z_n$ with respective weights $\alpha_1,\alpha_2, \ldots, \alpha_n$ if and only if there exists a supporting functional $f \in (\ell_1^n)^*$ at $(z-wy)$ such that $f(y) =0,$ where $z=(\alpha_1 z_1, \alpha_2 z_2, \ldots, \alpha_n z_n), y=(\alpha_1, \alpha_2, \ldots, \alpha_n) \in \ell_1^n.$ Therefore, given $z, y\in \ell_1^n,$ it is required to find a complex number $w$ such that $f(y)=0$ and $ f(z-wy) = \|f\| \|z-wy\|,$ for some bounded linear functional $f$ on $\ell_1^n.$  In this context, we  need the following lemma.

	\begin{lemma}\label{lemma1}
		Let $f$ be a supporting functional at  $ u=(u_1, u_2, \ldots, u_n) \in S_{\ell_1^n}$  such that 
 		\[ f(v_1,v_2, \ldots, v_n) = x_1v_1+v_2v_2 + \ldots + x_nv_n, \forall v = (v_1,v_2, \ldots, v_n) \in \ell_1^n,\]
		for some complex numbers $x_1,x_2, \ldots, x_n.$  
		If  $|x_{i}| <1$ for some $i$ then $u_{i}=0.$
	\end{lemma}
	
	\begin{proof} Since $f$ is a supporting functional at $u,$ it follows that 
		$f(u) = x_1u_1 + x_2u_2 + \ldots + x_nu_n = 1 $ and so $|x_i| \leq 1$ for all $i.$ 
		Let $N(x) $ denote the collection of all $i$ such that $|x_i| <1,$ i.e., $N(x) = \{ i \in \{ 1,2, \ldots, n\} : |x_i| <1\}.$ We claim that $u_i =0$ for each $i \in N(x).$ 
		Suppose on the contrary that there exists $ j \in N(x)$ such that  $ u_j \neq 0.$  Then we get
		\[
		1= 
		|\sum_{i=1}^{n} x_i u_i | \leq \sum_{i \neq j} | x_i u_i| +  |x_j u_j| < \sum_{i=1}^{n}|u_i|=1.\] 
		This establishes our claim.
	\end{proof}

	\section*{ Weighted Fermat-Torricelli Problem For 3 points}

	We now focus our attention to  the weighted Fermat-Torricelli problem of three points $z_1, z_2, z_3$ with respective weights $ \alpha_1, \alpha_2, \alpha_3,$ where each $\alpha_i > 0.$ First we prove the following two elementary observations.

	\begin{prop}\label{first}
		Let $z_1, z_2, z_3$ be  distinct complex numbers and let $\alpha_1, \alpha_2, \alpha_3 > 0.$  
		\begin{itemize}
			\item[(i)]  Suppose  $\alpha_i >  \alpha_j+ \alpha_k,$ for some  $i,j,k \in \{1,2,3\}.$  Then $z_i$ is the unique weighted Fermat-Torricelli point for $z_1, z_2, z_3$ with respective weights $ \alpha_1, \alpha_2, \alpha_3.$
			
			\item[(ii)]  Suppose  $\alpha_i = \alpha_j+ \alpha_k,$ for some  $i,j,k \in \in \{1,2,3\}.$  If $ z_j  \in L[z_i,z_k] $ $($ or $ z_k \in L[z_i,z_j] )$   then  every element $ w \in L[z_i,z_j]
			$ $($ or $ L[z_i,z_k] )$   is a  weighted Fermat-Torricelli point for $z_1, z_2, z_3$ with respective weights $ \alpha_1, \alpha_2, \alpha_3.$ Otherwise, $z_i$ is the unique weighted Fermat-Torricelli point. 
		\end{itemize}

	\end{prop}
	
	\begin{proof}
		(i) Without loss of generality we assume that $\alpha_3 >  \alpha_1+ \alpha_2.$
		For any $ x ( \neq z_3) \in \mathbb{C} ,$
		\begin{eqnarray*}
			\alpha_1|z_1-x| + \alpha_2 |z_2-x|+ \alpha_3 |z_3-x| & >& \alpha_1| z_1-x|    + \alpha_2 | z_2-x| +(\alpha_1+ \alpha_2)|z_3-x| \\
			&=& \alpha_1 ( |z_1-x|+|z_3-x|)+ \alpha_2 (|z_2-x|+|z_3-x|)\\
			&\geq & \alpha_1 |z_1-z_3|+ \alpha_2 |z_2-z_3|.
		\end{eqnarray*}
		This implies that $z_3$ is the  unique weighted Fermat-Torricelli point.
		
		(ii)  Without loss of generality we assume that $\alpha_3 =  \alpha_1+ \alpha_2$  and $z_2 \in L[z_1,z_3].$ Then for each $w \in L[z_2,z_3]$ we get,
		\begin{eqnarray*}
			\alpha_1|z_1-w| + \alpha_2 |z_2-w|+ \alpha_3 |z_3-w| & = & \alpha_1| z_1-w|    + \alpha_2 | z_2-w| +(\alpha_1+ \alpha_2)|z_3-w| \\
			&=& \alpha_1 ( |z_1-w|+|z_3-w|)+ \alpha_2 (|z_2-w|+|z_3-w|)\\
			&=  & \alpha_1 |z_1-z_3|+ \alpha_2 |z_2-z_3|.
		\end{eqnarray*}
		Again for any $ x \notin L[z_2,z_3] $ we get,  
		$	|z_2-x|+ |z_3-x|> |z_2-z_3|. $
		Therefore, 
		\begin{eqnarray*}
			\alpha_1|z_1-x| + \alpha_2 |z_2-x|+ \alpha_3 |z_3-x| & = & \alpha_1| z_1-x|    + \alpha_2 | z_2-x| +(\alpha_1+ \alpha_2)|z_3-x| \\
			&=& \alpha_1 ( |z_1-x|+|z_3-x|)+ \alpha_2 (|z_2-x|+|z_3-x|)\\
			&>  & \alpha_1 |z_1-z_3|+ \alpha_2 |z_2-z_3|.
		\end{eqnarray*}
		Thus each   $w \in L[z_2,z_3]$  is a weighted Fermat-Torricelli point and no other point is a weighted Fermat-Torricelli point. If $ z_1 \in L[z_2,z_3]$ then similarly we can show that  each   $w \in L[z_1,z_3]$  is a weighted Fermat-Torricelli point and no other point is a weighted Fermat-Torricelli point.\\
		Next assume that $\alpha_3 =  \alpha_1+ \alpha_2$  but  $z_2  \notin L[z_1,z_3] $ and $z_1 \notin L[z_2,z_3].$ 
		Then  for any $x \neq z_3$  at least one of the following inequalities are true
		\begin{eqnarray*}
			|z_1-x|+ |z_3-x|> |z_1-z_3|,\\
			|z_2-x|+ |z_3-x|> |z_2-z_3|.
		\end{eqnarray*}
		Then using similar arguments as in the above case, we can easily show that $x$ is not a weighted Fermat-Torricelli point and $z_3$ is the only weighted Fermat-Torricelli point.
	\end{proof}

	\begin{prop}\label{3-points}
		Let $z_1, z_2, z_3$ be  distinct complex numbers such that $z_i, z_j$ are unimodular for some distinct $i,j \in \{1,2,3\}. $ Let $\alpha_1, \alpha_2, \alpha_3 > 0$ be such that $ \alpha_i  < \alpha_j+ \alpha_k,$ for every distinct $i,j,k \in \{1,2,3\}.$ Assume that  $ \alpha_1 z_1+ \alpha_2 z_2 + \alpha_3 z_3=0.$
		\begin{itemize}
			\item[(i)] If $|z_k|<1,$ where $k \in \{1,2,3\} \setminus \{i,j\}$   then  $\angle( \overrightarrow{z_i}, \overrightarrow{z_j} )= \theta,$ where $\cos \theta < \frac{\alpha_k^2- \alpha_i^2- \alpha_j^2}{2 \alpha_i \alpha_j}.$
			
			\item[(ii)] If  $|z_k|=1,$ where $k \in \{1,2,3\} \setminus \{i,j\}$ then $\angle( \overrightarrow{z_i}, \overrightarrow{z_j} )= \theta, \angle( \overrightarrow{z_i}, \overrightarrow{z_k} )= \phi$ where $\cos \theta = \frac{\alpha_k^2- \alpha_i^2- \alpha_j^2}{2 \alpha_i \alpha_j},$ $ \cos \phi = \frac{\alpha_j^2- \alpha_i^2- \alpha_k^2}{2 \alpha_i \alpha_k}.$
		\end{itemize}
	\end{prop}
	
	\begin{proof}
		
		Let $\angle( \overrightarrow{z_i}, \overrightarrow{z_j} )= \theta.$ This implies $z_j= z_i e^{i \theta},$ where $0 \leq \theta < 2 \pi.$ 
		
		(i)  Clearly $|z_k|< 1 \implies |-\frac{\alpha_i}{\alpha_k} z_i- \frac{\alpha_j}{\alpha_k} z_j| < 1 \implies \alpha_i^2 +  \alpha_j^2+ 2 \alpha_i \alpha_j \cos \theta<  \alpha_k^2 \implies \cos \theta< \frac{\alpha_k^2- \alpha_i^2- \alpha_j^2}{2 \alpha_i \alpha_j}. $ This proves (i).
		
		(ii)	Proceeding as (i),   $|z_k|=1 \implies    \cos \theta=  \frac{\alpha_k^2- \alpha_i^2- \alpha_j^2}{2 \alpha_i \alpha_j}. $ Suppose that $\angle( \overrightarrow{z_i}, \overrightarrow{z_k} )= \phi.$ Then  using $|z_j|=1 $ and following similar  arguments as before, we get that $  \cos \phi= \frac{\alpha_j^2- \alpha_i^2- \alpha_k^2}{2 \alpha_i \alpha_k} .$ This completes the proof.
	\end{proof}
	
	\begin{remark}
		Whenever $ \alpha_1= \alpha_2= \alpha_3=1 $ and $z_1, z_2, z_3$ are unimodular distinct complex number such that $z_1+ z_2+ z_3=0,$ then $ z_1, z_2, z_3$ form the vertices of an equilateral triangle in the complex plane.
	\end{remark}

	As mentioned previously, Birkhoff-James orthogonality techniques are our main tool for analyzing the Fermat-Torricelli problem. In particular, it follows from Proposition \ref{Fermat} that to solve the weighted Fermat-Torricelli problem in the complex plane for three points, we require an explicit description of Birkhoff-James orthogonality in $ \ell_1^3. $ This is accomplished in the following result.
	
	\begin{theorem}\label{Fermat-Torricelli set 3}
		Let $ \mathbb{X}= \ell_1^3 $ and let $\alpha_1 , \alpha_2, \alpha_3 > 0 $ be such that $ \alpha_i  < \alpha_j+ \alpha_k,$ for every distinct $i,j,k \in \{1,2,3\}.$  If   $( c_1, c_2, c_3) \perp_B ( \alpha_1, \alpha_2, \alpha_3),$ then $( c_1, c_2, c_3) $ is  one of the following types:\\
		\begin{itemize}
			\item[Type I:] 
			
			$	(a)  \lambda( \mu, 0, 0),$\\
			$	(b) \lambda( 0, \mu,  0),$ \\
			$	(c) \lambda( 0, 0, \mu),$\\
			where $ \mu$ is unimodular and $\lambda \geq 0.$\\
			
			\item[ Type II:]
			$	(a)  \lambda( t\mu, (1-t) \sigma, 0) $ such that $  \angle (\overrightarrow{\sigma}, \overrightarrow{\mu}) = \theta$ and $cos \theta \leq \frac{\alpha_3^2 -\alpha_1^2 - \alpha_2^2}{2\alpha_1\alpha_2},$  \\
			$	(b) \lambda( 0, t\mu, (1-t)\sigma)$ such that $\angle (\overrightarrow{\sigma}, \overrightarrow{\mu}) = \theta$ and $cos \theta \leq \frac{\alpha_1^2 -\alpha_2^2-\alpha_3^2}{2\alpha_2\alpha_3},$   \\
			$	(c) \lambda( t\mu, 0, (1-t)\sigma)$ such that $ \angle (\overrightarrow{\sigma}, \overrightarrow{\mu}) = \theta$ and $cos \theta \leq \frac{\alpha_2^2 -\alpha_1^2 - \alpha_3^2}{2\alpha_1\alpha_3},$ \\
			where $\mu, \sigma$ are unimodular and $t \in (0,1)$, $\lambda \geq 0.$ \\
			
			\item[Type III:]
			$ \lambda (t_1 \mu, t_2 \sigma, t_3 \gamma ): \lambda \geq 0, t_1, t_2, t_3 \in (0, 1), t_1+ t_2 +t_3=1, $ where  $\angle (\overrightarrow{\sigma}, \overrightarrow{\mu}) = \theta,$
			and  $ \angle (\overrightarrow{\gamma}, \overrightarrow{\mu}) = \phi,$ where $\cos \theta= \frac{\alpha_3^2 -\alpha_1^2 - \alpha_2^2}{2\alpha_1\alpha_2} $ and $\cos \phi= \frac{\alpha_2^2- \alpha_3^2- \alpha_1^2}{2 \alpha_1 \alpha_3}.$\\
		\end{itemize}
	\end{theorem}
	
	\begin{proof}
		Let $ \alpha = ( \alpha_1, \alpha_2, \alpha_3) \in \ell_1^3.$ Our aim is to find all those $c = (c_1,c_2,c_3) $ in $\ell_1^3$ such that $ c \perp_B \alpha.$ 
		Since  Birkhoff-James orthogonality is homogeneous, it suffices to find those $c$'s  for which $ c \in S_{\mathbb{X}}.$   Theorem \ref{James} asserts that $ c \perp_B \alpha$ if and only if there exists a supporting functional $f$ at $c$  so that $ f(c)= \| c\|$ and $ f(\alpha) = 0.$ We first consider the set
		\[ A:= \{f \in (\ell_1^3)^* : f( \alpha) = 0 \} \] 
		which is same as 
		\[	A:=\{ x = (x_1, x_2, x_3) \in S_{\ell_\infty^3}:  f(\alpha)= x_1\alpha_1 + x_2 \alpha_2 + x_3 \alpha_3=0\}.	\]
		Now for any $x=(x_1, x_2, x_3)\in A,$ we define the set
		\[
		B(x):= \{ u = (u_1, u_2, u_3) \in S_{\mathbb{X}}: x_1 u_1+ x_2u_2+ x_3 u_3=1\}.\]
		Observe that the collection of  elements in $S_{\mathbb{X}}$ that are orthogonal to $\alpha$ is simply the collection of all those elements in $B(x),$ for every $x \in A.$
		For each $x \in A,$ we define  the set $N(x)$ as the collection of  $i$'s such that $|x_i| < 1$  i.e., 
		\[ N(x) = \{ i \in \{ 1,2,3\} : |x_i| < 1\}.\]
		Then clearly $ |N(x)| =0,1 $ or $2.$ We now consider the following three cases:
		
		Case I:  $ | N(x)| =2.$  Without loss of generality we assume  that $|x_1|=1.$ Clearly, $|x_2|, |x_3| < 1.$ Then for any $u  \in B(x),$ we have $x_1 u_1+ x_2 u_2+ x_3u_3=1$  and so from Lemma \ref{lemma1}, we get $u_2= u_3=0, u_1= \overline{x}_1.$ So, $(u_1, u_2, u_3) =(\overline{x}_1, 0, 0)$. Therefore,  it is easy to observe that $ (c_1, c_2, c_3) $ is of the form $\lambda( \mu, 0, 0),$ where $\mu$ is unimodular and $\lambda \geq 0.$	 Proceeding similarly, we can show that $(c_1, c_2, c_3)$ is of the form $\lambda (0 , \mu, 0)$ or $\lambda(0, 0, \mu)$ when $|x_2| =1$ or $|x_3|=1,$  respectively. 
		
		Case II: $ | N(x)| =1.$ Without loss of generality we assume  that $|x_1|=|x_2|=1.$ Clearly, $|x_3| < 1.$ From Proposition \ref{3-points}, we obtain that $ \angle (\overrightarrow{x_1}, \overrightarrow{x_2}) = \theta,$ where $cos \theta < \frac{\alpha_3^2- \alpha_1^2- \alpha_2^2}{2 \alpha_1 \alpha_2}.$ Now for any $u \in B(x), $ $x_1 u_1+ x_2 u_2+ x_3 u_3=1$  and so from Lemma \ref{lemma1}, we get $u_3=0.$  We observe that 
		$	| x_1 u_1| + | x_2 u_2| = |u_1| + |u_2|=1 = |x_1u_1+ x_2 u_2| .$ 
		If one of $u_1$ or $u_2$ is zero, we get first type of points. If both $u_1$ and $u_2$ are non-zero, then using the strict convexity of $\mathbb{C}$ we get $x_2u_2= p x_1 u_1,$ for some $p > 0.$ Therefore,
		\[
		u_1= \frac{1}{1+p} \overline{x}_1, ~  u_2= \frac{p}{1+p} \overline{x}_2.
		\] 
		In particular, 
		\[
		(u_1, u_2, u_3) = ( t \mu, (1-t) \sigma, 0), ~~~ t \in (0, 1),
		\]
		where $ \mu= \overline{x_1}, \sigma= \overline{x_2},$ which  are unimodular. Since    $ \angle (\overrightarrow{x_1}, \overrightarrow{x_2}) = \theta \implies \angle (\overrightarrow{\overline{x_1}}, \overrightarrow{\overline{x_2}})= 2\pi-\theta,$ it implies that $\angle (\overrightarrow{\sigma}, \overrightarrow{\mu})=\theta,$ where $\cos\theta < \frac{\alpha_3^2- \alpha_1^2- \alpha_2^2}{2 \alpha_1 \alpha_2}.$ So, in this case $(c_1, c_2, c_3)$ is one of the Type II points.  The result follows similarly for the other sub-cases.
		
		Case III:  $ |N(x)| = 0 .$  Clearly $|x_1| = | x_2| = |x_3| =1.$ Then $ \alpha_1 x_1 + \alpha_2 x_2 + \alpha_3 x_3 = 0.$
		From Proposition \ref{3-points}, we obtain that $\angle (\overrightarrow{x_1}, \overrightarrow{x_2}) = \theta$ and $  \angle (\overrightarrow{x_1}, \overrightarrow{x_3}) = \phi,$ where $cos \theta = \frac{\alpha_3^2- \alpha_1^2- \alpha_2^2}{2 \alpha_1 \alpha_2}$ and $cos \phi = \frac{\alpha_2^2- \alpha_1^2- \alpha_3^2}{2 \alpha_1 \alpha_3}.$ Now for any $u \in B(x),$ $x_1 u_1+ x_2 u_2+ x_3 u_3=1.$
		If any two of the $u_i$'s are zero then we obtain the first type of points. Let us now observe the case where exactly one of the $u_i,$ say $ u_3 $ is zero. Then  
		\[
		| x_1 u_1| + | x_2 u_2|  = | u_1| + |u_2| = 1 = | x_1 u_1 + x_2 u_2 |
		\] 
		Using the strict convexity of $\mathbb{C},$ similar to the previous case, we obtain $x_2u_2= p x_1 u_1,$ for some $p>0.$ Therefore,
		\[
		u_1= \frac{1}{1+p} \overline{x}_1, ~  u_2= \frac{p}{1+p} \overline{x}_2.
		\] 
		In particular, 
		\[
		(c_1, c_2, c_3)= \lambda (u_1, u_2, u_3) = \lambda( t \mu, (1-t) \sigma, 0), ~~~ t \in (0, 1),
		\]
		where $ \mu= \overline{x_1}, \sigma= \overline{x_2},$ which  are unimodular and $\lambda> 0$. Since $\angle (\overrightarrow{x_1}, \overrightarrow{x_2}) = \theta,$  where $cos \theta = \frac{\alpha_3^2- \alpha_1^2- \alpha_2^2}{2 \alpha_1 \alpha_2},$ we obtain  $  \angle (\overrightarrow{\sigma}, \overrightarrow{\mu}) = \theta.$ Therefore, $(c_1, c_2, c_3)$ is one of the Type II points and similarly by choosing the other $u_i$ as $0,$ we get the other Type II points.

		Assume that $ u_i \neq 0$ for each $i =1,2,3.$ 
		Observe that
		\[
		| x_1 u_1| + | x_2 u_2| + | x_3 u_3| = | u_1| + |u_2| + |u_3|= 1 = | x_1 u_1 + x_2 u_2 + x_3 u_3|.\]
		Using the strict convexity of $\mathbb{C},$  we get $x_2 u_2 = p x_1 u_1, ~ x_3 u_3 = q x_1 u_1,$ for some $p,q > 0.$ Therefore, 
		\[
		u_1= \frac{1}{1+p+q} \overline{x}_1, ~ u_2 = \frac{p}{1+p+q} \overline{x}_2, ~ u_3= \frac{q}{1+p+q} \overline{x}_3.
		\]
		In particular, 
		\[
		(u_1, u_2, u_3)= (t_1 \mu, t_2 \sigma, t_3 \gamma), ~ t_1, t_2, t_3 \in (0,1),~  t_1+ t_2+ t_3=1, 
		\] 
		where $ \mu= \overline{x_1}, \sigma= \overline{x_2}, \gamma= \overline{x_3},$ which are unimodular. 
		Since,  $ \angle (\overrightarrow{x_1}, \overrightarrow{x_2}) = \theta, $ $  \angle (\overrightarrow{x_1}, \overrightarrow{x_3}) = \phi ,$ so we get
		$ \angle (\overrightarrow{\overline{x_1}}, \overrightarrow{\overline{x_2}})= 2\pi-\theta,~ \angle (\overrightarrow{\overline{x_1}}, \overrightarrow{\overline{x_3}})= 2\pi-\phi.$ Thus we get,  $\angle (\overrightarrow{\sigma}, \overrightarrow{\mu})=\theta,~ \angle (\overrightarrow{\gamma}, \overrightarrow{\mu})=\phi,$ where $\cos\theta = \frac{\alpha_3^2- \alpha_1^2- \alpha_2^2}{2 \alpha_1 \alpha_2}$ and $\cos\phi =  \frac{\alpha_2^2- \alpha_1^2- \alpha_3^2}{2 \alpha_1 \alpha_3}.$
		Thus in this case $(c_1,c_2,c_3)$ is one of the Type III points. 
		This completes the theorem.  \\
	\end{proof}
	
	The above result allows us to solve the weighted Fermat-Torricelli problem in the complex plane for three points.
	
	\begin{theorem}\label{second}
		Let $z_1, z_2, z_3$ be  distinct complex numbers and let $\alpha_1, \alpha_2, \alpha_3 > 0.$ Suppose that $  \alpha_i  < \alpha_j+ \alpha_k,$ for every distinct $i,j,k \in \{1,2,3\}.$ 
		Let $w$ be the weighted Fermat-Torricelli point of  $z_1, z_2, z_3,$ where the weights are $\alpha_1, \alpha_2, \alpha_3,$ respectively. Then exactly one of following conditions is true: 
		
		\begin{itemize}
			\item[(i)] The weighted Fermat-Toricelli point $w$ coincides with one of the points, say,  $z_i$ and  $\angle (\overrightarrow{z_iz_j}, \overrightarrow{z_i z_k})= \theta,$ where $\cos \theta \leq \frac{\alpha_i^2- \alpha_j^2- \alpha_k^2}{2 \alpha_j \alpha_k}.$ 
			\item[(ii)]  $\angle (\overrightarrow{w z_j}, \overrightarrow{w z_i}) =\theta$ and $ \angle (\overrightarrow{ w z_k}, \overrightarrow{w z_i}) = \phi,$  where $\cos \theta = \frac{\alpha_k^2 -\alpha_i^2 - \alpha_j^2}{2\alpha_i\alpha_j} $ and $\cos \phi= \frac{\alpha_j^2- \alpha_i^2- \alpha_k^2}{2 \alpha_i \alpha_k}.$
		\end{itemize}

	\end{theorem}

	\begin{proof}

		Since $w$ is the weighted Fermat-Torricelli point of $z_1, z_2, z_3,$ then following Proposition \ref{Fermat}, we conclude that 
		\[ 
		(\alpha_1(z_1-w), \alpha_2(z_2-w), \alpha_3(z_3-w)) \perp_B (\alpha_1, \alpha_2 ,\alpha_3),
		\]
		in the Banach space $\ell_1^3.$
		Then from Theorem \ref{Fermat-Torricelli set 3}, $ 	(\alpha_1(z_1-w), \alpha_2(z_2-w), \alpha_3(z_3-w))$ is any one of the three types mentioned in Theorem \ref{Fermat-Torricelli set 3}.  We first claim that $ 	(\alpha_1(z_1-w), \alpha_2(z_2-w), \alpha_3(z_3-w))$ can not be of Type I. Indeed, assuming on the contrary that it is so, we have 
		$	(\alpha_1(z_1-w), \alpha_2(z_2-w), \alpha_3(z_3-w))= \lambda(\mu, 0, 0),$ for some unimodular $\mu \in \mathbb{C}$ and $\lambda \geq 0.$  This implies that 
		$ z_2=w=z_3,$ a contradiction to our hypothesis that $z_1,z_2,z_3$ are distinct  complex numbers. Next assume that  $ 	(\alpha_1(z_1-w), \alpha_2(z_2-w), \alpha_3(z_3-w))$ is of the Type II and    assume that $	(\alpha_1(z_1-w), \alpha_2(z_2-w), \alpha_3(z_3-w))= \lambda(t\mu, (1-t) \sigma , 0),$ for some unimodular complex numbers $\mu, \sigma \in \mathbb{C}, t \in (0,1) $ and $ \lambda \geq 0.$ Then  $ z_3=w$ and it follows from Theorem \ref{Fermat-Torricelli set 3} that,   $   \angle (\overrightarrow{z_3z_2}, \overrightarrow{z_3z_1}) =  \angle (\overrightarrow{z_2-w}, \overrightarrow{z_1-w})= \angle (\overrightarrow{\sigma}, \overrightarrow{\mu})= \theta,$ where $ \cos \theta \leq  \frac{\alpha_3^2- \alpha_1^2- \alpha_2^2}{2 \alpha_1 \alpha_2}.$  Proceeding similarly we can conclude that if $ 	(\alpha_1(z_1-w), \alpha_2(z_2-w), \alpha_3(z_3-w))$ is of the Type II, then the weighted Fermat-Toricelli point $w$ satisfies (i). 
		Finally let  $ 	(\alpha_1(z_1-w), \alpha_2(z_2-w), \alpha_3(z_3-w))$ be of the Type III. Without loss of generality assume that   $	(\alpha_1(z_1-w), \alpha_2(z_2-w), \alpha_3(z_3-w)) = \lambda ( t_1 \mu, t_2 \sigma, t_3 \gamma), $ where $ \mu, \sigma, 
		\gamma $ are unimodular complex numbers and $ t_1, t_2, t_3 \in (0,1 )$ such that $t_1+ t_2 + t_3=1.$ 
		From Theorem \ref{Fermat-Torricelli set 3}, we obtain that $ \angle (\overrightarrow{z_2-w}, \overrightarrow{z_1-w}) = \angle (\overrightarrow{\sigma}, \overrightarrow{\mu}) =\theta$ and $  \angle (\overrightarrow{z_3-w}, \overrightarrow{z_1-w}) = \angle (\overrightarrow{\gamma}, \overrightarrow{\mu})= \phi,$ where  $\cos \theta = \frac{\alpha_3^2 -\alpha_1^2 - \alpha_2^2}{2\alpha_1\alpha_2} $ and $\cos \phi= \frac{\alpha_2^2- \alpha_1^2- \alpha_3^2}{2 \alpha_1 \alpha_3}.$ This implies $\angle (\overrightarrow{w z_2}, \overrightarrow{w z_1}) =\theta$ and $ \angle (\overrightarrow{ w z_3}, \overrightarrow{w z_1}) = \phi.$ 	Therefore, in this case $w$ satisfies (ii). This completes the proof of the theorem.\\
		% Moreover, it is  straightforward to check that if  $\angle (\overrightarrow{z_3 z_1}, \overrightarrow{z_3 z_2}) = \zeta,$ then $\cos \zeta > \frac{\alpha_3^2- \alpha_1^2- \alpha_2^2}{2 \alpha_1 \alpha_2}.$
		%	Moreover such $w$ is unique.
		\\
	\end{proof}

	\begin{remark}
		Proposition \ref{first} and Theorem \ref{second}  together  cover all the possibilities for (positive) weights.
	\end{remark}
	
	Putting $\alpha_i=1$ in the above theorem, we have the following corollary.
	\begin{cor}\label{original}
		Let $z_1, z_2, z_3$ be  distinct complex numbers and 	let $w$ be the  Fermat-Torricelli point of  $z_1, z_2, z_3.$ Then exactly one of followings are true: 
		
		\begin{itemize}
			\item[(i)] The Fermat-Toricelli point $w$ coincides with one of the points, say, $z_i$ and $\angle (\overrightarrow{z_iz_j}, \overrightarrow{z_i z_k}) \geq \frac{2 \pi}{3}.$ 
			
			\item[(ii)] 
			$\angle (\overrightarrow{w z_j}, \overrightarrow{w z_i}) =\frac{2 \pi}{3}$ and $ \angle (\overrightarrow{ w z_k}, \overrightarrow{w z_i}) = \frac{4 \pi}{3},$ for some distinct $i,j,k \in \{1,2,3\}.$ \\
		\end{itemize}
	\end{cor}

	\section*{The Fermat-Torricelli Problem For 4-points}

	The Fermat-Torricelli problem for four points $z_1,z_2,z_3,z_4$ in the complex plane asks to find a point $w$ such that  for all complex numbers $x,$ 
	\[ |z_1-w| + |z_2 - w|\ + |z_3 - w| + |z_4-w| \leq |z_1-x| + |z_2-x| + |z_3-x| + |z_4-x| .\]
	
	The points $z_1,z_2,z_3,z_4$ are considered to be distinct and so the quadrilateral $Q$ formed by them is non-degenerate. We recall that a quadrilateral  with four vertices $z_1,z_2,z_3,z_4$ is called a convex quadrilateral if $z_i \notin co(\{z_j, z_k, z_l\}),$ for any distinct $i, j, k, l \in \{1, 2, 3, 4\}.$	In a convex quadrilateral, each interior angle  is lesser than $\pi.$  On the other hand, in a non-convex quadrilateral at least one angle is a reflex angle, i.e., an angle greater than or equal to  $\pi.$
	We require the following lemma for our purpose.

	\begin{lemma}\label{4-points}
		Let $z_1, z_2$ and  $z_3$ be 
		unimodular distinct complex numbers. Then 
		\begin{itemize}
			\item[(i)]  The three points  $z_i, z_j, 0\,(i \neq j) $ are collinear if and only if  $|z_1 + z_2 + z_3| =1.$ 
			
			\item[(ii)] The four points $ 0, z_1, z_2, z_3$ form  the vertices of a non-convex quadrilateral  if and only if $ |z_1+ z_2+ z_3| < 1.$ Also, in such a case,  $0 \in co(\{z_1, z_2, z_3\}). $  
		\end{itemize}
	\end{lemma}	
	
	\begin{proof}
		Let us assume that $ \angle (\overrightarrow{z_1}, \overrightarrow{z_2}) =\theta_1$ and $ \angle (\overrightarrow{z_1}, \overrightarrow{z_3}) =\theta_2, $   where $ 0 \leq \theta_1, \theta_2 <   2\pi.$
		Then  $z_2= z_1e^{i\theta_1}$ and $z_3= z_1e^{i\theta_2}.$ Therefore, $|z_1+ z_2+ z_3|= 1 \iff |1+ e^{i\theta_1}+ e^{i\theta_2} | =1 \iff cos \theta_1+ cos \theta_2 + cos(\theta_2- \theta_1) + 1= 0.$ By a straightforward calculation, we can easily observe that   
		\[
		|z_1+ z_2+ z_3|= 1 \iff cos \frac{\theta_1}{2} cos \frac{\theta_2}{2} cos \frac{\theta_2-\theta_1}{2}=0, \]
		which is equivalent to one of $ \theta_1, \theta_2$ and $(\theta_2-\theta_1)$ being equal to $\pi.$ Therefore, the proof of (i) follows directly.
		
		We next prove (ii). Observe that $|z_1+ z_2+ z_3|< 1 $ if and only if  exactly one of $\theta_1, \theta_2$ and $(\theta_2- \theta_1) $ is greater than $\pi.$ Hence   $|z_1+ z_2+ z_3|< 1 $ if and only if the four points $0, z_1, z_2, z_3$ form the vertices of a non-convex quadrilateral. It is clear that in this case  $0 \in co(\{z_1, z_2, z_3\}). $ 
	\end{proof}
	
	\begin{remark}\label{remark:rectangle}
		Let $z_1, z_2, z_3, z_4 $ be unimodular distinct complex numbers such that $z_1+ z_2 +z_3+ z_4=0.$ From the above lemma, it is easy to observe $z_1, z_2, z_3, z_4$ form the vertices of a rectangle.
	\end{remark}
	
	We next obtain the analogous result to Theorem \ref{Fermat-Torricelli set 3}, corresponding to the case $ n = 4. $

	\begin{theorem}\label{Fermat-Torricelli set 4}
		Let $ \mathbb{X}= \ell_1^4.$  If  $( c_1, c_2, c_3, c_4) \perp_B (1, 1, 1, 1),$ then $( c_1, c_2, c_3, c_4) $ is any one of the following types:

		\begin{enumerate}
			\item[Type I:] 
			
			$(a)  \lambda( \mu, 0, 0, 0),$\\
			$(b) \lambda( 0, \mu, 0, 0), $\\
			$(c) \lambda( 0, 0, \mu,  0),$\\
			$(d) \lambda( 0, 0, 0, \mu)$,\\
			where $ \mu$ is unimodular and $\lambda \geq 0.$\\

			\item[Type II:]
			
			$(a)  \lambda( t\mu, (1-t)\sigma, 0, 0),$\\
			$(b) \lambda( t \mu, 0, (1-t)\sigma, 0), $\\
			$(c) \lambda( t\mu , 0, 0,  (1-t)\sigma),$\\
			$(d) \lambda( 0, t\mu , (1-t)\sigma, 0)$,\\
			$(e) \lambda( 0, t\mu , 0, (1-t)\sigma)$,\\
			$(f) \lambda( 0, 0,  t\mu , (1-t)\sigma)$,\\
			where $ \mu, \sigma $ are unimodular, $t \in (0, 1)$ and $\lambda \geq 0.$\\
			
			\item[ Type III:]
			
			$(a)  \lambda( t_1 \mu, t_2 \sigma, t_3 \gamma, 0),$\\
			$(b) \lambda( 0, t_1 \mu, t_2 \sigma, t_3 \gamma),$ \\
			$(c) \lambda( t_1 \mu, 0,t_2 \sigma, t_3 \gamma),$\\
			$(d) \lambda( t_1 \mu, t_2 \sigma, 0, t_3  \gamma),$\\
			where $ \mu, \sigma, \gamma $ are unimodular, $\lambda \geq 0,t_1, t_2, t_3 \in (0, 1)$ such that $t_1+ t_2 + t_3=1$ and $ 0 \in co(\{ \mu, \sigma, \gamma\}).$\\
			
			\item[Type IV:]
			$\lambda( t_1 \mu, t_2 \sigma,  t_3  \gamma, t_4 \delta ),$  where $ \mu, \sigma, \gamma, \delta $ are unimodular, $\lambda \geq 0,$ and  $t_1, t_2, t_3, t_4 \in (0, 1)$ such that $t_1+ t_2 + t_3+ t_4=1.$ Moreover, $ \mu, \sigma, \gamma, \delta$ form the vertices of a rectangle in the complex plane.
			
		\end{enumerate}
	\end{theorem}

	\begin{proof}
		Let $ \alpha = ( 1,1,1,1) \in \ell_1^4.$ Our aim is to find all such $c = (c_1,c_2,c_3,c_4) \in \ell_1^4 $ such that $ c \perp_B \alpha.$ 
		As before, it suffices to find those $c$'s  for which $ c \in S_{\mathbb{X}}.$ We further note that $ c \perp_B \alpha$ if and only if there exists a supporting functional $f$ at $c$  so that $ f(c)= \| c\|$ and $ f(\alpha) = 0.$ We first consider the set
		\[ A:= \{f \in (\ell_1^4)^* : f( \alpha) = 0 \} \] 
		which is same as 
		\[	A:=\{ x = (x_1, x_2, x_3,x_4) \in S_{\ell_\infty^4}:  f(\alpha)= x_1 + x_2  + x_3 + x_4=0\}.	\]
		Now for any $x\in A,$ we define the set
		\[
		B(x):= \{ u = (u_1, u_2, u_3,u_4) \in S_{\mathbb{X}}: \sum_{i=1}^4 u_ix_i=1\}.\]
		Observe that the collection of  elements in $S_{\mathbb{X}}$ that are orthogonal to $\alpha$ is the collection of all those elements in $B(x),$ for every $x \in A.$
		For each $x \in A,$ we define  the set $N(x)$ as the collection of  $i$'s such that $|x_i| < 1$  i.e., 
		\[ N(x) = \{ i \in \{ 1,2,3,4\} : |x_i| < 1\}.\]
		Then clearly $ |N(x)| =0,1,2 $ or $3.$ We now consider the following four cases:

		Case I:  $|N(x)|=3.$ Without loss of generality assume that   $|x_1|=1.$ Now  for any $u \in B(x),$  we have $x_1 u_1+ x_2 u_2+ x_3u_3 + x_4u_4=1$ with $ |x_i| <1, $ for $  i=2,3,4$ and so from Lemma \ref{lemma1}, we get $u_2= u_3=u_4=0,$ and  $u_1= \overline{x}_1.$ So, $(u_1, u_2, u_3, u_4)$ is of the form $(\overline{x}_1, 0, 0, 0)= (\mu, 0, 0, 0),$ where $\mu$ is unimodular.  Thus $ (c_1, c_2, c_3,c_4) $ is of the form $\lambda( \mu, 0, 0,0),$ where $\mu$ is unimodular and $\lambda \geq 0.$	
		Similarly we obtain the rest of the Type-I points.
		
		Case II: $|N(x)|=2. $  Without loss of generality, we assume that $|x_1|=|x_2|=1.$ Then  $|x_3| < 1, |x_4| < 1.$  Again using Lemma \ref{lemma1},  we get that $ u_3=0,u_4=0$ if $ u \in B(x).$ Then proceeding similarly as in Theorem \ref{Fermat-Torricelli set 3},  we note that
		\[
		u_1= \frac{1}{1+p} \overline{x}_1, ~  u_2= \frac{p}{1+p} \overline{x}_2,
		\] 
		where $x_2u_2= p x_1 u_1,$ for some $p > 0.$ 	In particular, 
		\[
		(u_1, u_2, u_3, u_4) = ( t \mu, (1-t) \sigma, 0, 0), ~~~ t \in (0, 1),
		\]
		where $ \mu, \sigma$ are unimodular.  Thus $ (c_1, c_2, c_3,c_4) $ is of the form $\lambda(t \mu, (1-t) \sigma, 0,0),$ where $\mu$ is unimodular and $\lambda \geq 0.$	 The result follows similarly for the other sub-cases.

		Case III:  $|N(x)|=1. $  Assume that $|x_1|=|x_2|=|x_3|=1$  and $ |x_4| <1.$  Using  Lemma \ref{lemma1},  we get that $u_4=0$ if $ u \in B(x).$   If any of the points $u_1, u_2, u_3$ are equal to zero, then we can proceed  as in Case I and Case II.  Let us now assume that $u_1, u_2, u_3$ all are non-zero. Now  $| x_1 +  x_2+ x_3|  = |x_4| < 1$  and so from Lemma \ref{4-points},  it follows that the four points $0, x_1, x_2, x_4$ form the vertices of a non-convex quadrilateral where $0 \in co(\{x_1, x_2, x_3\}).$  Now 
		\[
		\bigg|\sum_{i=1}^{3} x_i u_i \bigg| = 1= \sum_{i=1}^{3} \bigg|u_i \bigg| = \sum_{i=1}^{3} \bigg|x_i u_i\bigg|.
		\]
		Proceeding as in Theorem \ref{Fermat-Torricelli set 3} and using the convexity of the complex plane $\mathbb{C},$ it is easy to verify that 
		\[	u_1= \frac{1}{1+p+q} \overline{x}_1, ~  u_2= \frac{p}{1+p+q} \overline{x}_2, ~ u_3= \frac{q}{1+p+q} \overline{x}_3, \]
		for some $p,q > 0.$
		In particular,
		\[	(u_1, u_2, u_3, u_4) = ( t_1 \mu, t_2 \sigma, t_3 \gamma, 0), 	\]	
		where $\mu, \sigma, \gamma$ are unimodular and $ t_1, t_2, t_3 \in (0, 1) $ with $   t_1+ t_2 + t_3=1.$ Since
		$0 \in co(\{x_1, x_2, x_3\})$ it is easy to observe that 
		$0 \in co(\{\mu, \sigma, \gamma\}).$  Thus $ (c_1,c_2,c_3,c_4) $ is of the desired  form of Type III.  The result follows similarly for the other sub-cases.

		Case IV:  $|N(x)|=0. $  Then  $|x_1| = | x_2| = |x_3| =|x_4|= 1.$ From Remark \ref{remark:rectangle}, it follows that the four points $x_1, x_2, x_3, x_4$ form the vertices of a rectangle.  Let $ u \in B(x).$ If one or more of $u_1, u_2, u_3, u_4$ is zero, then proceeding as before we get points of the Type I, II or III. 
		Let us now consider the case where $u_1, u_2, u_3, u_4$ all are non-zero.
		Then we have, 
		\[
		u_1= \frac{\overline{x}_1}{1+p+q+r} , ~ u_2 = \frac{p~ \overline{x}_2}{1+p+q+r} , ~ u_3=  \frac{q ~ \overline{x}_3}{1+p+q+r}, ~ u_4=\frac{r~ \overline{x}_4}{1+p+q+r},
		\]	
		for some $p, q, r> 0.$
		In particular,	
		\[	(u_1, u_2, u_3, u_4)= (t_1 \mu, t_2 \sigma, t_3 \gamma, t_4 \delta), ~ t_1, t_2, t_3, t_4\in (0,1),~  t_1+ t_2+ t_3+t_4=1, 
		\] 	
		where $ \mu, \sigma, \gamma, \delta$ are unimodular. From the Remark
		\ref{remark:rectangle},  $ \mu, \sigma, \gamma, \delta$  form the vertices of a rectangle. Thus in this case the points are of the form of Type IV. 
		This completes the proof. 
	\end{proof}
	
	We are now ready to solve the Fermat-Torricelli problem in the complex plane for four points, by using the same method that was employed in case of three points
	
	\begin{theorem}\label{FT:4points}
		Let   $z_1, z_2, z_3, z_4$ be distinct complex numbers and let  $w \in \mathbb{C}$ be the Fermat-Torricelli point of  $z_1, z_2, z_3, z_4.$ Then exactly one of the following conditions is true: 
		\begin{itemize}
			\item[(i)]  $w = z_l$ for some $l$ and  $z_l \in co(\{z_i, z_j, z_k\}).$ 
			
			\item[(ii)]  $w $ is the intersection of the diagonals of the quadrilateral formed by the vertices $z_1, z_2, z_3, z_4.$
		\end{itemize}
	\end{theorem}
	
	\begin{proof}
		
		Since $w$ is the Fermat-Torricelli point of $z_1,z_2,z_3,z_4,$ so we have $ (z_1-w, z_2-w, z_3-w, z_4-w) \perp_B (1,1,1,1) $ in the Banach space $\ell_1^4.$    From Theorem \ref{Fermat-Torricelli set 4}, it follows that  $ (z_1-w, z_2-w, z_3-w, z_4-w)$ is one of the four types mentioned in Theorem \ref{Fermat-Torricelli set 4}. As $z_1, z_2, z_3, z_4$ are all distinct, so  $ (z_1-w, z_2-w, z_3-w, z_4-w)$ can  not be of Type I and Type II. We discuss the remaining two cases as follows.
		
		Case I: Let $ (z_1-w, z_2-w, z_3-w, z_4-w)$ be  of the Type III. Without loss of generality we  may and do  assume that $(z_1-w, z_2-w, z_3-w, z_4-w)= \lambda(t_1\mu, t_2 \sigma , t_3 \gamma, 0),$ for some unimodular complex numbers $\mu, \sigma , \gamma \in \mathbb{C}$ and $\lambda \geq 0,~ t_1, t_2, t_3 \in (0,1)$ such that $t_1+ t_2 +t_3 =1.$ Then $ w=z_4.$ Also, from  Theorem \ref{Fermat-Torricelli set 4}, it follows that  $0 \in co(\{\mu, \sigma, \gamma\})$ and so    $ 0 \in co(\{ \frac{1}{\lambda t_1} (z_1-w), \frac{1}{\lambda t_2} (z_2-w), \frac{1}{\lambda t_3} (z_3-w)\}).$  Then it is easy to verify that $z_4 \in co(\{z_1,z_2,z_3\})$ and $w$ satisfies (i).

		Case II: Let  $ (z_1-w, z_2-w, z_3-w, z_4-w)$  be of  the Type IV. As before, it follows that $(z_1-w, z_2-w, z_3-w, z_4 -w) = \lambda ( t_1 \mu, t_2 \sigma, t_3 \gamma, t_4 \delta), $ where $ \mu, \sigma, 	\gamma, \delta $ are unimodular complex numbers and $ t_1, t_2, t_3, t_4 \in (0,1 )$ such that $t_1+ t_2 + t_3+ t_4=1.$  Also, 
		$ \frac{1}{\lambda t_1} (z_1-w), \frac{1}{\lambda t_2} (z_2-w), \frac{1}{\lambda t_3} (z_3-w), \frac{1}{\lambda t_4} (z_4-w) $ form the vertices of a rectangle and $0$ is the intersection of the diagonals of the rectangle. Without loss of generality we may and do assume  that the points $ \frac{1}{\lambda t_1} (z_1-w), 0, \frac{1}{\lambda t_3} (z_3-w) $ as well as  $ \frac{1}{\lambda t_2} (z_2-w), 0, \frac{1}{\lambda t_4} (z_4-w) $ are collinear.  Then it follows that $ z_1, w, z_3$ and $z_2, w, z_4$ are collinear. Therefore, $w$ is the intersection of the diagonals of the quadrilateral formed by the vertices $z_1, z_2, z_3, z_4.$ This establishes the theorem. 
	\end{proof}

	\begin{remark}
		As mentioned in the introduction, explicit solutions to the Fermat-Torricelli problem in the Euclidean plane are already well-known \emph{only} in the case of three points and four points. The readers are referred to \cite{GT,HZ, S96}  for three points, and to \cite{P,  Z2} for four points. The important thing to observe in this context is the fact that the methods proposed in these articles vary greatly depending upon the number of points involved. On the other hand, the method presented in this article remains the same in both the cases and essentially reduces the problem to an orthogonality problem in either $ \ell_1^3 $ or $ \ell_{1}^4. $ %Moreover, it is also clear that an explicit description of the Birkhoff-James orthogonality in $ \ell_1^n $ will lead to a complete solution to the 
	\end{remark}

  The following result allows us to transfer the Fermat-Torricelli problem into a orthogonality problem in $\ell_{1}^n$. The proof is omitted  as it follows directly from Proposition \ref{Fermat}.

	\begin{theorem}
		Let $z_1, z_2, \ldots, z_n$ be  distinct complex numbers. Then $w$ is the Fermat-Torricelli point of $z_1, z_2, \ldots, z_n$ if and only if $(z_1-w, z_2-w, \ldots, z_n-w) \in  {^\perp(1, 1, \ldots, 1)}$ in the space $\ell_{1}^n.$
		\end{theorem}
	
	 To determine the Fermat-Torricelli point for $n$ distinct given points in the complex plane explicitly, we require a complete description of the set $^\perp(1, 1, \ldots, 1)$ in $\ell_{1}^n.$ This has been obtained in the present section for $n=3,4.$ Similarly, in order to solve the weighted Fermat-Torricelli problem for $n$ distinct points in complex plane, we require a complete description of Birkhoff-James orthogonality in $\ell_{1}^n.$

	\section*{Some Properties of the Fermat-Torricelli point}
	
	In this section, we study the behavior of the weighted Fermat-Torricelli point  under the addition or replacement of a point.  We require the following basic observations for our purpose.

	\begin{prop}\label{prop:non-smooth}
		Let $\mathbb{X}=\ell_1^n$ and let $( \neq 0) v=(v_1, v_2, \ldots, v_n) \in \mathbb{X}.$ Suppose that $f(u)= \sum_{i=1}^{n}  d_i u_i,$ for any $u=(u_1, u_2, \ldots, u_n) \in \ell_1^n$ is  a support functional at $v.$ If for some $i\in \{ 1, 2, \ldots, n\},$ $v_i\neq 0,$ then $d_i= \frac{\overline{v_i}}{|v_i|}.$
	\end{prop} 
	
	\begin{prop} \label{proposition:smooth points}
		Let $\mathbb{X}=\ell_1^n.$ Then  ${v}=(v_1, v_2, \ldots, v_n) \in \mathbb{X}  $ is a smooth point if and only if $v_i \neq 0,$ for each $i = 1, 2, \ldots, n.$ Moreover, in that case, the unique supporting functional $f$ at ${v}$ is given by     $f( x) =   \sum_{i=1}^{n} x_i \frac{\overline{v}_i}{|v_i|},\, \forall \, x = (x_1,x_2,\ldots,x_n) \in \mathbb{X}.$ 
	\end{prop} 
	
	We are now ready to prove the main result of this section.
	
	\begin{theorem}\label{z_i=w1}
		Let $ z_1, z_2, \ldots, z_n$ be distinct  complex numbers and let $ \alpha_1, \alpha_2, \ldots, \alpha_n > 0. $ Suppose that $ w \in \mathbb{C}$ is  a weighted Fermat-Torricelli point of $z_1, z_2, \ldots, z_n  $ with respective  weights  $ \alpha_1, \alpha_2, \ldots, \alpha_n$   and $z_i \neq w,$ for any $ 1 \leq i \leq n.$ 
		\begin{itemize}
			\item[(i)]  For any $ \alpha_{n+1} > 0$ and $ z_{n+1} \in \mathbb{C}, $ $ w $ is a weighted  Fermat-Torricelli point of  $ z_1, z_2, \ldots, z_{n+1}$ with respective weights  $\alpha_1, \alpha_2, \ldots,  \alpha_{n+1}$   if and only if $z_{n+1}= w.$
			
			\item[(ii)]  Let $ s_n \in \mathbb{C}.$ Then  $ w $ is a weighted Fermat-Torricelli point of the points $ z_1, z_2, \ldots, z_{n-1} , s_n$ with respective weights  $ \alpha_1, \alpha_2, \ldots, \alpha_n$  if and only if 	$s_n -w= c(z_n -w),$ for some $c \geq 0.$
			
			\item[(iii)]   Let $ s_1,s_2, \ldots, s_m \in \mathbb{C}$ and $ \beta_i > 0 $ for $ i=1,2,\ldots,m.$  Then $ w $ is a weighted Fermat-Torricelli point of the points $ z_1, z_2, \ldots, z_n, s_1, s_2, \ldots , s_m $ with respective weights $\alpha_1, \alpha_2, \ldots, \alpha_n, \beta_1, \beta_2, \ldots, \beta_m$ if and only if $ w $ is a weighted Fermat-Torricelli point of $s_1, s_2, \ldots, s_m$ with respective weights $ \beta_1, \beta_2, \ldots, \beta_m.$
			
			\item[(iv)]  Let $ s_1,s_2,\ldots, s_n \in \mathbb{C}.$  Then  $ w $ is a weighted Fermat-Torricelli point of the points $  s_1, s_2, \ldots , s_n$ with respective weights  $ \alpha_1, \alpha_2, \ldots, \alpha_n$ if $s_i -w= c_i(z_i -w),$ where $ c_i <0,$ for each $i =1, 2, \ldots, n$ or $c_i > 0,$ for each $i = 1, 2, \ldots, n.  $
		\end{itemize}
		
	\end{theorem}

	\begin{proof}
		
		Since $w \in \mathbb{C}$ is a weighted Fermat-Torricelli point of $z_1, z_2, \ldots, z_n $  with respective weights  $ \alpha_1, \alpha_2, \ldots, \alpha_n,$ so we have,  $ x \perp_B y $ in $ \ell_1^n,$  
		where $x = (\alpha_1(z_1-w), \alpha_2(z_2-w),\ldots, \alpha_n(z_n-w)),  y = (\alpha_1, \alpha_2, \ldots, \alpha_n) \in \ell_1^n.$ Now, $z_i \neq w,$ for each $ i =1,2, \ldots,n$ and so from Proposition \ref{proposition:smooth points},  it follows that $ x $ is a smooth point of $\ell_1^ n.$  Also, the supporting functional $f$ at $x$ is given by $ f(u) =   \sum_{i=1}^{n} u_i 	\frac{\overline{z_i-w}}{|z_i-w|}, \forall \, u = (u_1,u_2,\ldots,u_n) \in \ell_1^n.$     Then applying Theorem \ref{James}, we get $y  \in ker f$  and so 
		\begin{eqnarray}\label{n}
			\alpha_1	\frac{\overline{z_1-w}}{|z_1-w|} + \alpha_2 \frac{\overline{z_2-w}}{|z_2-w|} + \ldots+ \alpha_n \frac{\overline{z_n-w}}{|z_n-w|} =0.
		\end{eqnarray}
		
		(i) The sufficient part follows trivially. We only prove the necessary part. Suppose on the contrary that $  z_{n+1} \neq w.$ 
		Proceeding as above  and noting that $ w$ is a weighted Fermat-Torricelli point of $ z_1, z_2, \ldots,  z_{n+1} $ with respective weights  $ \alpha_1, \alpha_2, \ldots,  \alpha_{n+1},$ we get, 
		\begin{eqnarray}\label{n+1}
			\alpha_1	\frac{\overline{z_1-w}}{|z_1-w|} + \alpha_2 \frac{\overline{z_2-w}}{|z_2-w|} + \ldots+ \alpha_n\frac{\overline{z_n-w}}{|z_n-w|} + \alpha_{n+1} \frac{\overline{z_{n+1}-w}}{|z_{n+1}-w|} =0.
		\end{eqnarray}
		From  (\ref{n}) and (\ref{n+1}) we get, 
		$ \frac{\overline{z_{n+1}-w}}{|z_{n+1}-w|} =0 $ and so $ z_{n+1}= w,$ a contradiction. 
		This completes the proof of (i).

		(ii) 	  We first prove the necessary part. If $s_n =w,$ then we have nothing to prove.  Let us assume that $s_n \neq w.$ Then using Theorem \ref{James}, we get that
		\begin{eqnarray}\label{replace}
			\alpha_1	\frac{\overline{z_1-w}}{|z_1- w|} + \alpha_2 \frac{\overline{z_2-w}}{|z_2-w|} + \ldots+ \alpha_{n-1} \frac{\overline{z_{n-1}-w}}{|z_{n-1}-w|} + \alpha_n \frac{\overline{s_n-w}}{|s_n-w|} =0.
		\end{eqnarray}
		From   (\ref{n}) and (\ref{replace}) we get, 
		\[
		\frac{\overline{z_n-w}}{|z_n-w|} =   \frac{\overline{s_n-w}}{|s_n-w|},
		\]
		which implies that $s_n-w = c(z_n-w),$ where $c= \frac{|s_n-w|}{|z_n-w|}>0. $ 
		
		Let us now prove the sufficient part. 
		By the hypothesis, $z_n-w= c (s_n-w),$ for some $c \geq 0.$ Since $ z_{n} \neq w $ we get, $c >0$ and $ s_n \neq w.$ Then from (\ref{n}) we get, 
		\begin{eqnarray}\label{replace2}
			\alpha_1	\frac{\overline{z_1-w}}{|z_1- w|} + \alpha_2 \frac{\overline{z_2-w}}{|z_2-w|} + \ldots+ \alpha_{n-1} \frac{\overline{z_{n-1}-w}}{|z_{n-1}-w|} + \alpha_n \frac{\overline{s_n-w}}{|s_n-w|} =0.
		\end{eqnarray} 
		
		Let  $ z = (\alpha_1(z_1-w), \alpha_2(z_2-w),\ldots,  \alpha_{n-1}(z_{n-1}-w), \alpha_n(s_n-w)), $ $  y = (\alpha_1, \alpha_2, \ldots, \alpha_n) \in \ell_1^n.$ Consider the linear functional $g$ defined on $\ell_1^n$ as
		\[ g(u) = u_1 \frac{\overline{z_1-w}}{|z_1- w|} + \ldots +  u_{n-1} \frac{\overline{z_{n-1}-w}}{|z_{n-1}-w|} + u_n \frac{\overline{s_n-w}}{|s_n-w|} ,\]
		for any $u=(u_1, u_2, \ldots, u_n) \in \ell_1^n.$
		It is easy to see that $g$ is the supporting functional at $z$ and from equation (\ref{replace2}), we obtain that $ g(y)=0.$ 
		From Theorem \ref{James}, it follows that $ z \perp_B y$  and hence using Proposition \ref{Fermat}, we conclude that $w$ is a  weighted Fermat-Torricelli point of $z_1, z_2, \ldots, z_{n-1}, s_n $  with respective weights  $ \alpha_1, \alpha_2, \ldots, \alpha_n.$ 
		
		(iii) Let us first prove the necessary part. Since $w$ is a weighted Fermat-Torricelli point of  $ z_1, z_2, \ldots, z_n, s_1, s_2, \ldots , s_m$ with respective weights $ \alpha_1, \alpha_2, \ldots, \alpha_n, $ $ \beta_1, \beta_2, \ldots, \beta_m,$ from  Proposition \ref{Fermat} and Theorem \ref{James} it follows that there exists a supporting functional $f$ at $ z \in \ell_1^{n+m}$  such that $ f(z) = \|z\|$ and $f(x_0)=0$, where $z = ( \alpha_1(z_1-w), \alpha_2(z_2-w), \ldots, \alpha_n(z_n-w),  \beta_1(s_1-w), \beta_2(s_2-w), \ldots , \beta_m(s_m-w)) , x_0 = (\alpha_1, \alpha_2, \ldots, \alpha_n, $ $ \beta_1, \beta_2, \ldots, \beta_m) \in \ell_1^{n+m}. $ 
		Then there exist complex numbers $d_1, d_2, \ldots, $ $ d_{n+m}$ such that for any  $ u = (u_1,  u_2, \ldots, u_{n+m}) \in \ell_1^{n+m}, $ $f(u)= \sum_{i=1}^{n+m} d_i u_i.$ As  $\|f\|=1,$ clearly, $\max\{|d_i|: 1 \leq i \leq n+m\}=1.$
		Since $z_i \neq w,$ for any $1 \leq i \leq n,$ from Proposition \ref{prop:non-smooth} it is easy to observe that $d_i= \frac{\overline{z_i-w}}{|z_i-w|},$ for any $i, 1 \leq i \leq n.$
		Now $f(x_0)=0$ implies that 
		\[
		\alpha_1	\frac{\overline{z_1-w}}{|z_1-w|} +  \ldots+ \alpha_n\frac{\overline{z_n-w}}{|z_n-w|} + 	\beta_1	d_{n+1} +  \ldots+ \beta_{m} d_{n+m}  =0.
		\]
		Then using   (\ref{n}), we get that 
		\begin{eqnarray}\label{new-m}
			\beta_1	d_{n+1}+  \beta_2 d_{n+2}+  \ldots+ \beta_m d_{n+m}  =0. 
		\end{eqnarray}
		Consider the linear functional $g$ on $\ell_1^m$ given by 
		$ g(u) = \sum_{i=1}^m u_id_{n+i}, \forall u = (u_1,\ldots,u_m) \in \ell_1^{m}.$ Since $ \max\{|d_{n+i}|: 1 \leq i \leq m\}\leq 1,$ it follows that $\|g\| \leq 1.$
		As $f(z)=\|z\|,$ it is easy to observe that $ g(z_0) = \|z_0\|,$  where $ z_0 = (\beta_1(s_1-w),  \ldots, \beta_m(s_m-w)) \in \ell_1^m$ and from equation (\ref{new-m}), $g (\beta_1, \beta_2, \ldots, \beta_m ) =0.$ 
		As before, using Theorem \ref{James} and Proposition \ref{Fermat}, we conclude that $w$ is the weighted Fermat-Torricelli point of $s_1, s_2, \ldots, s_m$ with respective weights $\beta_1, \beta_2, \ldots, \beta_m.$ 
		
		Next, we  prove the sufficient part. Since $w$ is the weighted Fermat-Torricelli point of $z_1, z_2, \ldots, z_n$ with the respective weights $\alpha_1, \alpha_2, \ldots, \alpha_n$, we obtain that the equation (\ref{n}) holds. As $w$ is also the weighted Fermat-Torricelli point of $s_1, s_2, \ldots, s_m$ with the respective weights $\beta_1, \beta_2, \ldots, \beta_m$, using Proposition \ref{Fermat} and Theorem \ref{James}, we can find a supporting functional $h$ at $z_0 \in \ell_1^m$  such that $h(y_0) =0,$ where 
		$ z_0 = (\beta_1(s_1-w), \beta_2(s_2-w), \ldots, \beta_m(s_m-w)) , y_0 = (\beta_1, \beta_2, \ldots, \beta_m ) \in \ell_1^m.$  Then there exists complex numbers $v_1,v_2,\ldots,v_m$ such that $h(u) = \sum_{i=1}^m u_iv_i, \forall u=(u_1,u_2, \ldots,u_m) \in \ell_1^m.$  Now we define a linear functional $\phi$ on $\ell_1^{n+m}$ as 
		\[ \phi(u) = \sum_{i=1}^n u_i \frac{\overline{z_i-w}}{|z_i-w|}+ \sum_{i=1}^m u_{n+i}v_i, \, \forall u=(u_1,\ldots,u_n,u_{n+1},\ldots,u_{n+m}) \in \ell_1^{n+m}.\]
		As $h(z_0)= \|z_0\|,$ it is easy to	observe that $ \phi(z) = \|z\|,.$ where $z = ( \alpha_1(z_1-w),  \ldots, \alpha_n(z_n-w),  \beta_1(s_1-w),  \ldots , \beta_m(s_m-w)) \in \ell_1^{n+m}.$ Using equation (\ref{n}) and the fact that $h(y_0)=0,$ we obtain that $\phi( \alpha_1, \ldots, \alpha_n, \beta_1, \ldots, \beta_m)=0.$ Thus using Theorem \ref{James} and Proposition \ref{Fermat} we conclude that $w$ is a weighted Fermat-Torricelli point of  $ z_1,  \ldots, z_n, s_1,  \ldots , s_m$ with respective weights $ \alpha_1,  \ldots, \alpha_n, $ $ \beta_1,  \ldots, \beta_m.$ This completes the proof.

		(iv) Assume that  $ s_i-w =c_i(z_i-w),$ where $c_i> 0$ for each $ i =1,2, \ldots,n.$  Then $	\frac{s_i-w}{| s_i-w|}= \frac{z_i-w}{|z_i-w|}.$
		Therefore, following equation (\ref{n}) and using the above condition, we can easily obtain
		\[
		\alpha_1	\frac{\overline{s_1-w}}{|s_1-w|} + \alpha_2\frac{\overline{s_2-w}}{|s_2-w|} + \ldots+ \alpha_n \frac{\overline{s_n-w}}{|s_n-w|} =0.\]
		Then the linear functional $f$ defined on $\ell_1^n$ by 
		$$ f(u) =   \sum_{i=1}^{n} u_i 	\frac{\overline{s_i-w}}{|s_i-w|}, \forall \, u = (u_1,u_2,\ldots,u_n) \in \ell_1^n, $$ 
		is a supporting functional at $ z_0 = (\beta_1(s_1-w),  \ldots, \beta_n(s_n-w)) $ of $\ell_1^n$ and $f(y_0)=0,$ where $ y_0 = (\beta_1, \beta_2, \ldots, \beta_n ) \in \ell_1^n.$ So from Theorem \ref{James} and Proposition \ref{Fermat}, it follows that  $ w $ is a weighted Fermat-Torricelli point of the points $  s_1, s_2, \ldots , s_n$ with respective weights  $ \alpha_1, \alpha_2, \ldots, \alpha_n.$ If $c_i <0$ for each $ i =1,2, \ldots,n,$ then the result follows similarly.
	\end{proof}

	So far we have only discussed the cases where the Fermat-Torricelli point $w$  is distinct from each of the initial points $z_i.$ If the Fermat-Torricelli point $w$ coincides with one of the points $z_i,$ then we have the following result.

	\begin{theorem}\label{z_i=w2} 
		Let $ z_1, z_2, \ldots, z_n$ be distinct  complex numbers and  let $ \alpha_1, \alpha_2, \ldots, \alpha_{n+1} > 0. $ Let $ w$ be a weighted  Fermat-Torricelli point of $ z_1, z_2, \ldots, z_n$ with respective weights $\alpha_1, \alpha_2, \ldots, \alpha_n$  and $z_{i_0}=w,$ for some $ i_0 \in \{1, 2, \ldots, n\}.$
		\begin{itemize}
			\item[(i)] If $\alpha_{n+1}> 2 \alpha_{i_0},$  then there does not exist $z_{n+1} \neq w$ such that $ w$ is a weighted Fermat-Torricelli point of $ z_1, z_2, \ldots, z_n, z_{n+1},$  where the weights are $\alpha_1, \alpha_2, \ldots, \alpha_n, \alpha_{n+1}.$
			
			\item[(ii)] If $\alpha_{n+1}\leq    \alpha_{i_0},$  then there  exists $z_{n+1} \neq w$ such that $ w$ is a weighted Fermat-Torricelli point of $ z_1, z_2, \ldots, z_n, z_{n+1},$ where the weights are $\alpha_1, \alpha_2, $ $ \ldots, \alpha_n, \alpha_{n+1}.$ 
		\end{itemize}
	\end{theorem}
	\begin{proof}
		Without loss of generality we assume that $w=z_1.$   Let $ x= (0, \alpha_1(z_2- z_1), \alpha_2 (z_3- z_1), \ldots, \alpha_n(z_n-z_1)) $ and  $ y = (\alpha_1, \alpha_2, \ldots, \alpha_n).$  Since $ z_1$ is  a weighted Fermat-Torricelli point of $ z_1, z_2, \ldots, z_n,$ where the weights are $\alpha_1, \alpha_2, \ldots, \alpha_n,$ so 
		following Theorem \ref{James} and Proposition \ref{Fermat}, we can find  a supporting functional  $f$ at $	x \in \ell_1^n$ such that $f(y)=0.$ From Proposition \ref{prop:non-smooth}, observe that 
		$$ f(u) = \gamma u_1 + \sum_{i=2}^n u_i \frac{\overline{z_i-z_1}}{|z_i-z_1|}, \, \forall u =(u_1,u_2, \ldots, u_n) \in \ell_1^n,$$ 
		for some complex number $\gamma $ with $ | \gamma | \leq 1.$  Then $f(y)=0$ implies that 
		\begin{equation}\label{eqn:w=z_1v2}
			\alpha_1	\gamma + \alpha_2  \frac{\overline{z_2-z_1}}{|z_2-z_1|} + \alpha_3 \frac{\overline{z_3-z_1}}{|z_3-z_1|}+ \ldots + \alpha_n \frac{ \overline{z_n-z_1}}{|z_n-z_1|} =0.
		\end{equation}
		
		(i) Let $\alpha_{n+1} > 2 \alpha_1.$ Suppose on the contrary that there exists  $z_{n+1} \neq z_1$ such that $ z_1$ is a weighted  Fermat-Torricelli point of $ z_1,  \ldots,  z_{n+1} $ with respective weights $\alpha_1,  \ldots,  \alpha_{n+1}.$ Then  there exists a supporting functional $g$ at $(0, \alpha_1(z_2- z_1), \alpha_2 (z_3- z_1), \ldots, \alpha_{n+1}(z_{n+1}-z_1)) \in \ell_1^{n+1}$ such that $g(\alpha_1, \alpha_2, \ldots, \alpha_{n+1})=0.$ This implies there exists $\delta \in \mathbb{C}$ with $|\delta|\leq 1$ such that 
		\[ g(u) = \delta u_1 +  \sum_{i=2}^{n+1} u_i \frac{\overline{z_i-z_1}}{|z_i-z_1|}, \, \forall u =(u_1,u_2, \ldots, u_n,u_{n+1}) \in \ell_1^{n+1}\]
		and 
		\begin{eqnarray*}
			\alpha_1	\delta + \alpha_2 \frac{\overline{z_2-z_1}}{|z_2-z_1|}+ \alpha_3 \frac{ \overline{z_3-z_1}}{|z_3-z_1|}+ \ldots+ \alpha_n \frac{\overline{z_n-z_1}}{|z_n-z_1|}+ \alpha_{n+1} \frac{\overline{z_{n+1}- z_1}}{|z_{n+1}- z_1|}=0.
		\end{eqnarray*}
		The above equation, along with equation (\ref{eqn:w=z_1v2}), implies that 
		\[
		\frac{\overline{z_{n+1}- z_1}}{|z_{n+1}- z_1|}= \frac{\alpha_1}{\alpha_{n+1}}( \gamma- \delta).
		\]
		As $|\gamma|\leq 1, |\delta| \leq 1, $
		\[
		1= \bigg|\frac{\overline{z_{n+1}- z_1}}{|z_{n+1}- z_1|}\bigg|= \bigg|\frac{\alpha_1}{\alpha_{n+1}}( \gamma- \delta)\bigg| \leq 2\bigg|\frac{\alpha_1}{\alpha_{n+1}}\bigg|< 1.
		\]
		This contradiction proves the first part of the theorem.
		
		(ii) Suppose that $\alpha_{n+1} \leq   \alpha_1.$ We can find   $\delta \in \mathbb{C}$ such that $|\delta| \leq 1$ and $| \delta- \gamma |=\frac{\alpha_{n+1}}{\alpha_1}.$  Take $z_{n+1}= z_1+\frac{\alpha_{1}}{\alpha_{n+1}} ( \overline{\gamma - \delta}).$ Then  $z_{n+1} \neq z_1.$  Consider the linear functional $h$ defined on $\ell_1^{n+1}$ by 
		$$ h(u) = \delta u_1 +  \sum_{i=2}^{n+1} u_i \frac{\overline{z_i-z_1}}{|z_i-z_1|}, \, \forall u =(u_1,u_2, \ldots, u_n,u_{n+1}) \in \ell_1^{n+1}.$$ 
		Then $h(x_0)=\|x_0\|, $ where $ x_0 = (0, \alpha_2(z_2- z_1), \alpha_3(z_3- z_1), \ldots,\alpha_{n+1}( z_{n+1}-z_1)) \in \ell_1^{n+1}$   and  so $h$ is a supporting functional at $x_0.$ Also, from equation (\ref{eqn:w=z_1v2})
		\begin{eqnarray*}
			\alpha_1	\delta &+& \alpha_2 \frac{\overline{z_2-z_1}}{|z_2-z_1|}+ \alpha_3 \frac{ \overline{z_3-z_1}}{|z_3-z_1|}+ \ldots+ \alpha_n \frac{\overline{z_n-z_1}}{|z_n-z_1|}+ \alpha_{n+1} \frac{\overline{z_{n+1}- z_1}}{|z_{n+1}- z_1|} \\
			&=& \alpha_1 (\delta- \gamma) +  \alpha_{n+1}\frac{\overline{z_{n+1}- z_1}}{|z_{n+1}- z_1|}\\ &=&0.
		\end{eqnarray*} 
		which shows that $ h(\alpha_{1}, \alpha_2, \ldots, \alpha_{n+1}) =0.$ 
		Thus from Theorem \ref{James} and Proposition \ref{Fermat}, it follows that $z_1$ is a weighted Fermat-Torricelli point of $z_1, z_2, \ldots,  z_{n+1} $  with respective  weights $\alpha_1, \alpha_2, $ $ \ldots, \alpha_{n+1}.$
	\end{proof}

	We end this section with the following remark.
	
	\begin{remark}
		Let $ z_1, z_2, \ldots, z_n, z_{n+1}$ be distinct  complex numbers and  let $ w \in \mathbb{C}$ be the  Fermat-Torricelli point of $z_1, z_2, \ldots, z_n. $  Then from Theorem \ref{z_i=w1} and Theorem \ref{z_i=w2}, it follows that  $z_i \neq w,$ for any $ i=1,2,\ldots,n$ if and only if  $ w $ is a Fermat-Torricelli point of  $ z_1, z_2, \ldots, z_n , z_{n+1}$ implies that  $z_{n+1}= w.$
	\end{remark}

	%************************************************************************************************************************************************************************************************************************************************************************
	
	\section{Section -II}
	\section*{Chebyshev Center Problem}
	
	The aim of this section is to study the Chebyshev center of finite number of points in Euclidean plane using the concept of Birkhoff-James orthogonality.  The uniqueness of the Chebyshev center follows from \cite{A2}, in fact the uniqueness characterizes uniformly convex Banach spaces.  Let us now rewrite the connection between Birkhoff-James orthogonality and Chebyshev center for finitely many complex numbers.
	
	\begin{prop}\label{prop:cheby}
		Let $z_1, z_2, \ldots, z_n$ be $n$-distinct complex numbers. Then  $w$ is the Chebyshev center of $z_1, z_2, \ldots, z_n$ if and only if 
		$(z_1-w, z_2-w, \ldots, z_n-w) \perp_B (1,1,\ldots, 1) $ in $\ell_\infty^n(\mathbb{C})$.
	\end{prop}

%	\begin{theorem}\cite{A2}
	%	A Banach space $\mathbb{X}$ is uniformly convex if and only if for any nonempty bounded subset $A$ of $\mathbb{X},$ the Chebyshev center is unique.
	%	\end{theorem}
	
%	Now since the complex plane $\mathbb{C}$ is a uniformly convex Banach space, then the solution of the Chebyshev problem given in Proposition \ref{prop:cheby} is unique.
	
We also need the following easy proposition which characterizes the $k$-smooth points in $\ell_{\infty}^n.$
	
	\begin{prop}\label{prop:k-smooth}
		Let $\widetilde{x}=(x_1, x_2, \ldots, x_n) \in \ell_{\infty}^n.$ Then $\widetilde{x}$ is a $k$-smooth point if and only if there exist $i_1, i_2, \ldots, i_k \in \{1, 2, \ldots, n\}$ such that $|x_{i_1}|= |x_{i_2}|= \ldots= |x_{i_k} | > |x_j|,$ for any $j \in \{1, 2, \ldots, n\} \setminus \{i_1, i_2, \ldots, i_n\}.$ Moreover, if $f \in J(\widetilde{x}),$ then $f(u_1, u_2, \ldots, u_n)= \sum_{j=1}^{k} t_j \frac{\overline{x_{i_j}}}{|x_{i_j}|} u_{i_j},$ where $t_j \geq 0, \sum_{j=1}^{k} t_j=1,$ for any $(u_1, u_2, \ldots, u_n) \in \ell_{\infty}^n.$
	\end{prop}

Now we are in a position to present a complete solution to the Chebyshev center problem for finitely many points in the complex plane.

\begin{theorem}\label{Cheby:necessary}
		Let $w$ be the Chebyshev center of the distinct complex numbers $z_1, z_2, \ldots, z_n$  and let $\widetilde{z}= (z_1-w, z_2-w, \ldots, z_n-w) .$
		\begin{itemize}
			\item[(i)] Then $\widetilde{z}$ is not  a smooth point in $\ell_{\infty}^n.$
			\item[(ii)] If $\widetilde{z}$ is a $k$-smooth point of $\ell_{\infty}^n,$ for some $k, 2 \leq k \leq n,$ then there exist distinct $i_1, i_2, \ldots, i_k \in \{1, 2, \ldots, n\}$ such that $w$ is the circumcenter of $z_{i_1}, z_{i_2}, \ldots, z_{i_k}$ and $w \in co(\{z_{i_1}, z_{i_2}, \ldots, z_{i_k}\}).$ Moreover, $|z_j-w|< |z_{i_r}-w|,$ for any $j \in \{1, 2, \ldots, n\} \setminus \{i_1, i_2, \ldots, i_k\}$ and $ r \in \{ 1,2, \ldots,k\}.$
		\end{itemize} 
	\end{theorem}	
	
	\begin{proof}
			(i)	Suppose on the contrary that $\widetilde{z}$ is a smooth point of $\ell_\infty^n.$  Then there exists $i, 1\leq i \leq n,$ such that $|z_i-w|>|z_j-w|,$ for any $j \in \{1, 2, \ldots, n\}\setminus \{i\}.$ Suppose that $f \in (\ell_{\infty}^n)^*$ such that $f(u_1, u_2, \ldots, u_n)= \frac{\overline{z_i-w}}{|z_i-w|} u_i,$ for any $(u_1, u_2, \ldots, u_n) \in \ell_{\infty}^n.$
		Clearly, $J(\widetilde{z})= \{f\} .$   Now using Theorem \ref{James}, $f(1,1,\ldots, 1)=0.$ This implies $z_i-w=0,$  which is a contradiction. Therefore $\widetilde{z}$ is not a smooth point of $\ell_\infty^n$. 
		
		(ii) Let $\widetilde{z}$ be a $k$-smooth point. From Proposition \ref{prop:k-smooth}, there exist distinct $i_1, i_2, \ldots, i_k \in \{1, 2, \ldots, n\}$  such that   $|z_{i_1}-w|=|z_{i_2}-w|=  \ldots= |z_{i_k}-w|>|z_j-w|,$ for any $j \in \{1, 2, \ldots,n\} \setminus \{i_1, i_2, \ldots, i_k\}.$ Let $f \in J(\widetilde{z}).$ Again from Proposition \ref{prop:k-smooth} we get that 
		$$f(u_1, u_2, \ldots, u_n)= \sum_{j=1}^{k} t_j \frac{\overline{z_{i_j}-w}}{|z_{i_j-w}|} u_{i_j}, \quad \forall (u_1, u_2, \ldots, u_n) \in \ell_{\infty}^n,$$ 
		for some $t_j \geq 0, \sum_{j=1}^{k} t_j=1.$  Since $w$ is the Chebyshev center of $z_1, z_2, \ldots, z_n,$  from Proposition \ref{prop:cheby}, we obtain that 
		$$ 	(z_1-w, z_2-w, \ldots, z_n-w) \perp_B (1,1,\ldots, 1), ~ \text{in ~$\ell_\infty^n$.}$$
		Therefore, using Theorem \ref{James},  it follows that $f(1, 1, \ldots, 1)=0,$ for some $f \in J(\widetilde{z}).$ This implies $\sum_{j=1}^{k} t_j \frac{\overline{z_{i_j}-w}}{|z_{i_j-w}|} =0,$ for some $t_j \geq 0$ and $\sum_{j=1}^{k} t_j=1.$ By a simple computation, this implies $w = \sum_{j=1}^{k} t_j z_{i_j}.$ So, $w \in co(\{z_{i_1}, z_{i_2}, \ldots, z_{i_k}\}).$ Since $|z_{i_1}-w|=|z_{i_2}-w|=  \ldots= |z_{i_k}-w|,$   we obtain that $w$ is the circumcenter of $z_{i_1}, z_{i_2}, \ldots, z_{i_k}.$ This completes the proof of the theorem.
	\end{proof}
	
	Conversely, we have the following result.
	
\begin{theorem}\label{Cheby:sufficient}
		Let $z_1, z_2, \ldots, z_n$ be distinct complex numbers. Suppose there exist distinct  $i_1, i_2, \ldots, i_k \in \{1, 2, \ldots, n\}$ such that the circumcenter $w$ of $z_{i_1}, z_{i_2}, \ldots, z_{i_k}$ lies in $co(\{ z_{i_1}, z_{i_2}, \ldots, z_{i_k}\}).$ If $|z_j-w|< |z_{i_r}-w|,$ for any $j \in \{1, 2, \ldots, n\} \setminus \{i_1, i_2, \ldots, i_k\}$ and $ r \in \{ 1,2, \ldots,k\}.$ then $w$ is the Chebyshev center of $z_1, z_2, \ldots, z_n.$  
\end{theorem}

\begin{proof}
	As $w$ is the circumcenter of $z_{i_1}, z_{i_2}, \ldots, z_{i_k},$ we have $|z_{i_1}-w|= |z_{i_2}-w|= \ldots= |z_{i_k}- w|.$ Again $|z_j-w|< |z_{i_r}-w|,$ for any $j \in \{1, 2, \ldots, n\} \setminus \{i_1, i_2, \ldots, i_k\}$ and $ r \in \{ 1,2, \ldots,k\}.$ So, from Proposition \ref{prop:k-smooth}, $ \widetilde{z}=(z_1-w, z_2-w, \ldots, z_n-w)$ is a $k$-smooth point. Since $w \in co(\{z_{i_1}, z_{i_2}, \ldots, z_{i_k}\}),$ assume that $w= \sum_{j=1}^{k} t_j z_{i_j},$ where $t_j \geq 0, \sum_{j=1}^{k} t_j=1.$ Consider 
	$$ f(u_1, u_2, \ldots, u_n)= \sum_{j=1}^{k} t_j \frac{\overline{z_i-w}}{|z_i-w|} u_{i_j}, \quad \forall (u_1, u_2, \ldots, u_n) \in \ell_{\infty}^n.$$ 
	Using Proposition \ref{prop:k-smooth}, it is clear that $f \in J(\widetilde{z}).$ Now observe that $f(1, 1, \ldots, 1)= \sum_{j=1}^{k} t_j \frac{\overline{z_{i_j}-w}}{|z_{i_j}-w|}= \frac{1}{|z_{i_1}-w|} (\sum_{j=1}^{k} t_j z_{i_j}-w)=0.$ From Theorem \ref{James} and Proposition \ref{prop:cheby} it follows that $w$ is the Chebyshev center of $z_1, z_2, \ldots, z_n.$
\end{proof}

Following a method similar to those used in the proofs of Theorems \ref{Cheby:necessary} and  \ref{Cheby:sufficient},  we can also solve the Chebyshev center problem for the weighted case.

\begin{theorem}
		Let $z_1, z_2, \ldots, z_n$ be distinct complex numbers and $\alpha_1, \alpha_2, \ldots, \alpha_n > 0. $ Then $w$ is the weighted Chebyshev center of the points $z_1, z_2, \ldots, z_n$ with the respective weights $\alpha_1, \alpha_2, \ldots, \alpha_n$ if and only if  there exist  distinct $i_1, i_2, \ldots, i_k \in \{1, 2, \ldots, n\}$  such that  $w \in co(\{z_{i_1}, z_{i_2}, \ldots, z_{i_k}\})$ and $ \alpha_1|z_{i_1}- w|= \alpha_2|z_{i_2}- w|= \ldots = \alpha_k|z_{i_k}- w|,$ where $2 \leq k \leq n.$  Moreover, $\alpha_j|z_j-w|< \alpha_1|z_{i_1}-w|,$ for any $j \in \{1, 2, \ldots, n\} \setminus \{i_1, i_2, \ldots, i_k\}.$
	 \end{theorem}

	The following example illustrates the computational effectiveness of Theorem \ref{Cheby:sufficient}, in explicitly finding the Chebyshev center of a finite number of distinct points in the complex plane.

	\begin{example}
		Let $z_1= 4+i, z_2= 1 + 2i, z_3= 2-i, z_4= 3+ (1+\sqrt{2})i, z_5= 2-\sqrt{3}.$ To find the Chebyshev center, from the Theorem \ref{Cheby:necessary}, we only need to consider the circumcenters of the $k$-points from $z_1, z_2, \ldots, z_5,$ for any $k, 2 \leq k \leq 5.$ Observe that this set of the circumcenters is a finite set.
	 Let $u_{ij}$ be the midpoint of the line segment $L[z_i, z_j],$ for any $1 \leq i \neq j \leq 5.$ Clearly, $u_{12}= \frac{5}{2} +  \frac{3}{2}i$ and $|u_{12}- z_1|= \frac{\sqrt{10}}{2} .$ But we can check that $|u_{12}-z_3| = \frac{\sqrt{26}}{2}> \frac{\sqrt{10}}{2},$ so from Theorem \ref{Cheby:necessary}, $u_{12}$ can not be the Chebyshev center. By a similar computation, we can check that $u_{ij}$ is not the Chebyshev center for any $1 \leq  i \neq j \leq 5.$  Next, taking the three points $z_1, z_2, z_3,$ we check that $w'=\frac{9}{4} + \frac{3}{4}i$ is the circumcenter of $z_1, z_2, z_3.$ But $|z_1- w'|< |z_4-w'|,$ so $w'$ can not be the Chebyshev center. Similarly, we check other three points. In particular, considering the points $z_1, z_3, z_4$ we observe that $w= 2+i$ is the circumcenter of $z_1, z_3, z_4.$ Now, $|z_1-w|=2> |z_2-w|$ but $|z_5-w|=2.$ This implies that $w$ is the circumcenter of $z_1, z_3, z_4, z_5$ and $w \in co(\{z_1, z_3, z_4, z_5\}).$ Also $|z_2-w|< |z_1-w|=2.$ So from Theorem \ref{Cheby:sufficient}, $w= 2 + i$ is the Chebyshev center of these five points. 
	\end{example}

%In the next result we characterize the cyclic polygons whose circumcenter lies inside the polygon. We should note that a cyclic polygon is a polygon with vertices upon which a circle can be circumscribed.	The  result follows immediately from the above given algorithm.
	
%	\begin{theorem}
%		Let $z_1, z_2, \ldots, z_n$ be  distinct complex numbers. Then the polygon formed by $z_1, z_2, \ldots, z_n$ as its vertices is cyclic and the circumcenter is inside the polygon if and only if $w$ is the Chebyshev center of $z_1, z_2, \ldots, z_n$ satisfy the Step-$(n-1).$
%	\end{theorem}

	In the following two results, we address the cases where the Fermat-Torricelli point and the Chebyshev center coincides for three and four points respectively.
	
	\begin{theorem}\label{coincide:three}
		Let $z_1, z_2, z_3$ be three distinct complex numbers such that $z_1, z_2, z_3$ are non-collinear. Then the Fermat-Torricelli point  and the Chebyshev center of $z_1, z_2, z_3 $ coincide if and only if $z_1, z_2, z_3$ form the vertices of an  equilateral triangle.
	\end{theorem}
	
	\begin{proof}
			Let us first prove the easier sufficient part. Let $w$ be the centroid of the triangle whose vertices are $z_1, z_2, z_3.$ It is immediate that $w$ is the circumcenter of $z_1, z_2, z_3$ and so from Theorem \ref{Cheby:sufficient}, it follows that  $w$ is the Chebyshev center of $z_1, z_2, z_3.$
		 Moreover, from Corollary \ref{original}, it is obvious that $w$ is the Fermat-Torricelli point of $z_1, z_2, z_3.$ This completes the sufficient part. 
		
		We next prove the necessary part. Let $w$ be the Fermat-Torricelli point as well as the Chebyshev center of $z_1, z_2, z_3.$ From Theorem \ref{Cheby:necessary}, it follows that $ \widetilde{z}=(z_1-w, z_2-w, z_3-w)$ is not a smooth point in $\ell_{\infty}^3.$ Then $ \widetilde{Z}$ is either a 2-smooth or a 3-smooth point. First we consider the possibility that $\widetilde{z}$ is a $2$-smooth point. Following Theorem \ref{Cheby:necessary}, there exist two points say $z_1, z_2$ such that $|z_1-w|=|z_2-w|$ and $w \in L[z_1, z_2].$ Therefore, it is immediate that $w= \frac{1}{2}(z_1+z_2).$ In other words, 
	 $w$ is the midpoint of the edge $L[z_1, z_2]$ of the triangle formed by the vertices $z_1, z_2, z_3$. But $w$ is the Fermat-Torricelli point of $z_1, z_2, z_3,$ therefore using Corollary \ref{original}, it is immediate that $w$ can not be the mid point of an edge of the triangle. Therefore, $\widetilde{z}$ is not a $2$-smooth point of $\ell_{\infty}^3.$ So, $\widetilde{z}$ is a $3$-smooth point and therefore, $|z_1-w|=|z_2-w|=|z_3-w|.$		
		Now  consider the triangle formed by the vertices $z_1, z_2$ and $w.$ Clearly, it is an isosceles triangle and the angle formed at  $w$ by the two sides $L[w,z_1]$ and $ L[w,z_2] $ is $\frac{2 \pi}{3}.$ So, the angle forms at the points $z_1$ and $z_2$ by the sides $ L[w,z_1], L[z_2,z_1]$ and $L[w,z_2], L[z_1,z_2],$ respectively is $\frac{\pi}{6}.$ Again considering the triangle formed by the vertices $z_1, z_3$ and $w$ and following similar arguments, we obtain that the angle formed at the points $z_1$ by the sides $L[w,z_1], L[z_2,z_1] $ is $\frac{\pi}{6}.$ Therefore, in the triangle $z_1, z_2, z_3$ the angle formed at $z_1$ is $\frac{\pi}{3}.$ Similarly, we can show that the other angles of the triangle  formed by the vertices  $z_1, z_2$ and $z_3$ are also $\frac{\pi}{3}.$ Thus the triangle formed by the vertices $z_1, z_2, z_3$ is an equilateral triangle. This completes the proof.
	\end{proof} 
	
	\begin{theorem}\label{coincide:four}
	Let $z_1, z_2, z_3, z_4$ be four distinct complex numbers forming the vertices of  a convex quadrilateral. Let $w$ be the intersection point of the two diagonals.  Then the Fermat-Torricelli point  and the Chebyshev center of $z_1, z_2, z_3, z_4 $ coincide if and only if there exist two opposite vertices $z_i, z_j$ such that $|z_i-w|=|z_j-w|\geq \max\{|z_k-w|: k \in \{1, 2, 3, 4\} \setminus \{i,j\}\}.$
\end{theorem}

\begin{proof}
	Since $w$ is the intersection point of the two diagonals, from Theorem \ref{FT:4points} it follows that $w$ is the Fermat-Torricelli point. Also, from Theorem \ref{Cheby:sufficient}, it is easy to observe that $w$ is the Chebyshev center of the points $z_1,z_2, z_3, z_4.$
	
	Let us now prove the necessary part. Since the Chebyshev center and the Fermat-Torricelli point coincides,  $w$ is the Chebyshev center of $z_1, z_2, z_3, z_4.$ Let $\widetilde{z}=(z_1-w, z_2-w, z_3-w, z_4-w).$ So, from Theorem \ref{Cheby:necessary}, $\widetilde{z}$ is not a smooth point. Now consider that $\widetilde{z}$ is a $2$-smooth point. Without loss of generality, assume that $|z_1-w|=|z_2-w|> \max \{|z_3-w|, |z_4-w|\}.$  Now for any $f \in J(\widetilde{z}),$ $f(u_1, u_2, u_3, u_4)= (1-t) \frac{\overline{z_1-w}}{|z_1-w|} u_1+ t  \frac{\overline{z_2-w}}{|z_2-w|} u_2, $ where $0 \leq t \leq 1.$  As $w$ is the Chebyshev center, $\widetilde{z} \perp_B (1, 1, 1, 1)$ and therefore, using Theorem \ref{James}, we obtain that $$ (1-t) \frac{\overline{z_1-w}}{|z_1-w|}+ t  \frac{\overline{z_2-w}}{|z_2-w|} =0,$$ for some $t \in [0,1].$  By a simple computation, we obtain that $w = \frac{1}{2} (z_1+z_2).$ Since $w$ is the intersection point of the diagonals and $w \in L[z_1, z_2],$ therefore $z_1, z_2$ are two opposite vertices. So, we obtain the required result.
	If $\widetilde{z}$ is a $3$-smooth point, then without loss of generality assume that $|z_1-w|=|z_2-w|=|z_3-w|>|z_4-w|.$ It is immediate that there are two vertices from $z_1, z_2, z_3$ which are opposite vertices. The result follows trivially when  $\widetilde{z}$ is a $4$-smooth point.
\end{proof}
	
	In the following remark we consider the case where the four points form the vertices of a non-convex quadrilateral.
	
	\begin{remark}
		Let $z_1, z_2, z_3, z_4$ be four distinct complex numbers such that $z_4 \in co(\{z_1, z_2, z_3\}).$ Then the Fermat-Torricelli point and the Chebyshev center coincide if and only if $z_4$ is the circumcenter of the points $z_1, z_2, z_3.$
	\end{remark}

	 As a direct corollary of Theorem \ref{coincide:four}, we record the following observation. 
	
	\begin{cor}
		Let $z_1, z_2, z_3, z_4$ be four distinct complex numbers such that $z_1, z_2, z_3, z_4$ form the vertices of a parallelogram. Then  the Fermat-Torricelli point  and the Chebyshev center of $z_1, z_2, z_3, z_4 $ coincide.
	\end{cor}

	We now present a remark in which we give a formula for the Chebyshev radius.
	
	\begin{remark}
		Let $\mathcal{D}_n$ be the space of all $n \times n$ complex diagonal matrices. By the notation $((a_1, a_2, \ldots, a_n)) \in \mathcal{D}_n$ we denote the diagonal matrix $A$ whose diagonal entries are $a_1, a_2, \ldots, a_n,$ respectively. It is rather straightforward to observe that $\ell_{\infty}^n(\mathbb{C})$ is isometrically isomorphic to $\mathcal{D}_n$ endowed with the usual operator norm.   Naturally, an  isometry between these two spaces is  $\psi: \ell_{\infty}^n \to \mathcal{D}_n$ defined as $\psi(x_1, x_2, \ldots, x_n) = ((x_1, x_2, \ldots, x_n)).$ Therefore, the orthogonality problem defined in Proposition \ref{prop:cheby} can also be presented as an orthogonality problem in $\mathcal{D}_n.$ Suppose that $z_1, z_2, \ldots, z_n$ are distinct complex numbers and $\alpha_1, \alpha_2, \ldots, \alpha_n > 0.$ Let $w$ be the weighted Chebyshev center of $z_1, z_2, \ldots, z_n$ with respective weights $\alpha_1, \alpha_2, \ldots, \alpha_n.$ Let $T= ((\alpha_1 z_1, \alpha_2 z_2, \ldots, \alpha_n z_n)), A=((\alpha_1, \alpha_2, \ldots, \alpha_n)) \in \mathcal{D}_n.$ Then from Proposition \ref{prop:cheby} and using the identification between $\ell_{\infty}^n$ and $\mathcal{D}_n,$ we obtain that  $w$ is the weighted  Chebyshev center if and only if $(T-wA) \perp_B A,$ in the space $\mathcal{D}_n.$ Therefore, the weighted  Chebyshev radius $r$ is equal to $dist(T, span\{A\}).$ Using Theorem \cite[Th. 5.2]{S21}, we get that
		\begin{eqnarray*}
			r= dist(T, span\{A\}) &=&  \sup \{|\langle Tx, y \rangle|: \|x\|=\|y\|=1, \langle Ax,y \rangle =0\} \\
			&=& \sup \bigg\{\bigg|\sum_{i=1}^{n} z_i x_i \overline{y_i} \bigg|:\sum_{i=1}^{n} |x_i|^2 = \sum_{i=1}^{n} |y_i|^2=1,  \sum_{i=1}^{n} \alpha_i x_i \overline{y_i}=0\bigg\}.
		\end{eqnarray*}
		If all the weights are $1,$ the  Chebyshev radius $r$  is given by 
		 $$    r = \sup \bigg\{\bigg|\sum_{i=1}^{n} z_i x_i \overline{y_i} \bigg|:\sum_{i=1}^{n} |x_i|^2 = \sum_{i=1}^{n} |y_i|^2=1,  \sum_{i=1}^{n} x_i \overline{y_i}=0\bigg\}. $$
	\end{remark}
	
We end this article with this following remark.
	
	\begin{remark}
		Let us consider $\mathbb{X}= \mathbb{R}^m$ and $\mathbb{Y}= \mathbb{R}^n,$ both endowed with the Euclidean norm, in the least square problem defined in (\ref{least:eqn}). Let $A$ be an $n \times m$ matrix and  let $z=(z_1, z_2, \ldots, z_n) \in \mathbb{R}^n.$  Then the problem reduces to:
		\[
		\underset{x \in \mathbb{R}^m}{\text{min}} \|Ax-z\| = \Big [\sum_{k=1}^{m} (a_i^Tx- z_i)^2 \Big]^{\frac{1}{2}},
		\]
		where $a_i^T$ is the $i$-th row of $A.$
	Let $w \in \mathbb{R}^n$ satisfy the above minimization problem. Then using Proposition \ref{prop:least}, $(Aw-z) \perp Ax,$ for any $x \in \mathbb{X}.$ Therefore, $\langle Aw-z, Ax \rangle=0  \implies \langle A^TA w- A^Tz, x \rangle = 0,$ for any $x \in \mathbb{R}^n.$ This implies $A^TAw= A^T z.$ So the least square problem is equivalent to solving the above matrix equation. If $rank~A= n,$ then we get  $w= (A^TA)^{-1} A^T z.$
	\end{remark}
	
	\noindent \textbf{Declaration of interests.}\\
		The authors declare that they have no known competing financial interests or personal relationships that could have appeared to influence the work reported in this paper.

%\noindent \textbf{Data availability statement.} \\
%Data sharing not applicable to this article as no datasets were generated or analysed during the current study.

\end{document}